\documentclass[10pt]{amsart}
\usepackage{amsthm}
\usepackage{amssymb,amsmath,amscd}
\usepackage[T1]{fontenc}

\usepackage{mathrsfs}

\usepackage[all]{xy}

\usepackage{setspace}

\usepackage{appendix}

\usepackage{graphicx}

\usepackage{lscape}

\newtheorem{theorem}{Theorem}[section]
\newtheorem{lemma}[theorem]{Lemma}
\newtheorem{proposition}[theorem]{Proposition}
\newtheorem{corollary}[theorem]{Corollary}
\newtheorem{conjecture}[theorem]{Conjecture}

\newtheorem{question}[theorem]{Question}

\newenvironment{definition}{\refstepcounter{theorem}\begin{trivlist}\item [\hskip \labelsep {\bfseries Definition \thetheorem}]}{\end{trivlist}}

\theoremstyle{remark}

\newtheorem{remark}[theorem]{Remark}

\newcommand{\Z}{\mathbb{Z}}

\newcommand{\R}{\mathbb{R}}
\newcommand{\C}{\mathbb{C}}

\newcommand{\F}{\mathbb{F}}
\newcommand{\PP}{\mathbb{P}}

\newcommand{\bk}{{\bf k}}
\newcommand{\zp}{\mathbb{Z}_{\geq 0}}

\newcommand{\moda}[3][{[A/\pm1]}]{\overline{M}\left(#1; #2; #3\right)}
\newcommand{\modac}[2]{M^\circ\left([A/\pm1]; #1; #2\right)}

\renewcommand{\tilde}{\widetilde}
\renewcommand{\hat}{\widehat}

\DeclareMathOperator{\map}{Map}
\DeclareMathOperator{\Km}{Km}

\begin{document}

\bibliographystyle{alpha}

\title[Hyperelliptic curves on Abelian surfaces]{Counting hyperelliptic curves on an Abelian surface with quasi-modular forms}
\author{Simon C. F. Rose}
\address{Department of Mathematics, University of British Columbia, Canada}
\email{scfr@math.ubc.ca}

\begin{abstract}
In this paper we produce a generating function for the number of hyperelliptic curves (up to translation) on a polarized Abelian surface using the crepant resolution conjecture and the Yau-Zaslow formula. We present a formula to compute these in terms of MacMahon's generalized sum-of-divisors functions, and prove that they are quasi-modular forms.
\end{abstract}

\maketitle

\tableofcontents

\section{Introduction}\label{chap:intro}

Let $(A_{h-1},L_{h-1})$ be a polarized abelian surface with polarization of type $(1, h - 1)$. Up to translation in $A_{h-1}$, there is an $(h - 2)$-dimensional family of curves of arithmetic genus $h$ in the homology class $c_1(L_{h-1})^\vee$. The codimension of the hyperelliptic locus in $\overline{M}_{h,0}$ is $h - 2$, and so the following natural question arises:

\begin{question}\label{q_main}
Given a polarized abelian surface $(A_{h-1},L_{h-1})$, how many curves (up to translation in $A_{h-1}$) of geometric genus $g$ in the class $c_1(L_{h-1})^\vee$ are hyperelliptic?
\end{question}

We will often write $A_{h-1}$ or $L_{h-1}$ simply as $A$ or $L$ if there is no possibility of confusion, and when the degree of the polarization is not important. We will also for convenience make the substitution
\[
n = h - 1
\]
which works to make most of the formul\ae{} cleaner. Furthermore, throughout this paper, whenever we say ``the number of curves in $A$...'' we will always be referring to the number of curves in the class $c_1(L)^\vee$ {\em up to translation in $A$}. Let $N_{g,h}$ denote the number of hyperelliptic curves of geometric genus $g$ and arithmetic genus $h$ in a fixed $(A,L)$. Let
\[
F_g(u) = \sum_{h=g}^\infty N_{g,h}u^{h - 1}
\]
be the generating function for these numbers. We will give an explicit formula for $F_g(u)$ in terms of quasi-modular forms.


\begin{remark}
Note that in the case $g = 2$, all curves are hyperelliptic. In \cite{gottsche}, it is shown that
\[
F_2(u) = \sum_{d=0}^\infty \sigma_1(d)u^d = E_2(u) + \frac{1}{24}
\]
where $\sigma_1(d) = \sum_{k \mid d} k$, and $E_2$ is the Eisenstein series of weight 2. Thus we see that $F_2$ is in the ring of quasi-modular forms.
\end{remark}

The goal of this paper is to transform this natural enumerative problem into the language of orbifold Gromov-Witten theory, and to use the crepant resolution conjecture \cite{crc} and the Yau-Zaslow formula \cite{yau_zaslow} to compute the generating functions $F_g$ for all $g$.

We should remark that this number $N_{g,h}$ is not necessarily well defined---that is, independent of the choice of $A$---nor necessarily finite. In Section \ref{sec_gw_hyper} we will interpret it in terms of Gromov-Witten invariants which will be defined for all polarized $A$ with $c_1(L)^\vee$ primitive; in the case that $A$ is sufficiently generic, we expect that this will coincide with the honest count of hyperelliptic curves of geometric and arithmetic genera $g$ and $h$, respectively.

In fact, we provide a refinement of this count. Let $A[2]$ denote the collection of 2-torsion points in $A$. As we will see in Section \ref{sec_gw_hyper}, we can translate a hyperelliptic curve so that all of its Weirstrass points all lie on points of $A[2]$. We can then use the number of Weirstrass points lying over each $v \in A[2]$ to refine our count, as follows. 

Let $\bk : A[2] \to \zp$ be a function, denote by $|\bk| = \sum_{v \in A[2]}\bk(v)$, and let $g$ be such that $2g + 2 = |\bk|$. Let $N_{\bk,h}$ denote the number of curves $C \subset A$ of geometric genus $g$ and arithmetic genus $h$ whose normalizations $\overline{C}$ are hyperelliptic and with $\bk(v)$ Weirstrass points lying over $v$ for each $v \in A[2]$ (See Section \ref{sec_gw_hyper}, equation \eqref{nkn} for a precise definition). Let $P$ be the 
collection of $v \in A[2]$ so that $\bk(v)$ is odd. 
The main theorem of this paper is the following.

\begin{theorem}\label{thm_main}
Assume the Gromov-Witten crepant resolution conjecture for the resolution $\Km(A) \to A/\pm1$ (See Section \ref{subsec_crc}). Then the generating function $F_{g,\bk}(u) = \sum_{h = g}^\infty N_{\bk,h}u^{h-1}$ is given by
\begin{equation}\label{eq_main}
F_{g,\bk}(u) = \displaystyle E(u)^{\tfrac{1}{2}|S| - 2}\prod_{v\in S}A_{\frac{\bk(v)-1}{2}}(u^4)\prod_{v \notin S} C_{\frac{\bk(v)}{2}}(u^2)
\end{equation}
when $P$ satisfies an easily-verified condition (see Remark \ref{rmk_condition}), and is zero otherwise. The functions $E(q), A_i(q)$, and $C_i(q)$ are given by
\begin{gather*}
E(q) = \sum_{k=0}^\infty \sigma_1(2k+1)q^{2k+1}\\
A_i(q) = \sum_{0 < m_1 < \cdots < m_i} \frac{q^{m_1 + \cdots + m_i}}{(1 - q^{m_1})^2 \cdots (1 - q^{m_i})^2} \\
C_i(q) = \sum_{0 < m_1 < \cdots < m_i} \frac{q^{2m_1 + \cdots + 2m_i-i}}{(1 - q^{2m_1-1})^2 \cdots (1 - q^{2m_i-1})^2}
\end{gather*}
which are all quasi-modular forms.
\end{theorem}

We provide for reference a few computations of these series, all pertaining to genus three curves. From the description above, the only functions which will contribute to curves of genus three are
\begin{gather*}
A_2(u^4) \qquad C_2(u^2) \qquad E(u)^2 \qquad A_1(u^4)C_1(u^2)\\
C_1(u^2)^2 \qquad E(u)A_1(u^4) \qquad E(u)C_1(u^2)
\end{gather*}
a few of whose coefficients are given in Table \ref{tbl_coeffs}. In particular, the function $F_3(u)$ counting all curves of genus 3 is given by
\begin{multline*}
F_3(u) = A_2(u^4) + 3 C_2(u^2) + 12A_1(u^4)C_1(u^2) + 21C_1(u^2)^2 \\+ 10 E(u)C_1(u^2) + 6 E(u) A_1(u^4) + 3 E(u)^2
\end{multline*}
which is also included in Table \ref{tbl_coeffs}.

\begin{table}[ht]
\begin{tabular}{c | c c c c c c c c c c c c c c c c}
& $q^2$ & $q^3$ & $q^4$ & $q^5$ & $q^6$ & $q^7$ & $q^8$ & $q^9$ & $q^{10}$ & $q^{11}$ & $q^{12}$  \\ \hline\hline
$E(u)^2$ & 1 & & 8 & & 28 & & 64 & & 126 & & 224 \\
$E(u)C_1(u^2)$ & & 1 & & 6 & & 18 & & 40 & & 75 & \\
$C_1(u^2)^2$ & & & 1 & & 4 & & 12 & & 24 & & 44 \\
$E(u)A_1(u^4)$ & & & & 1 & & 4 & & 9 & & 20 & \\
$A_1(u^4)C_1(u^2)$ & & & & & 1 & & 2 & & 7 & & 10 \\
$C_2(u^2)$ & & & & & & &  1 & & 2 & & 4 \\
$A_2(u^4)$ & & & & & & & & & & & 1 \\ \hline
$F_3(u)$ & 3 & 10 & 45 & 66 & 180 & 204 & 471 & 454 & 972 & 870 & 1729 
\end{tabular}
\caption[Coefficients for genus 3 curves]{Some coefficients of the generating functions for genus 3 curves.}\label{tbl_coeffs}
\end{table}

The structure of the paper will be as follows. In Section \ref{chap_preliminaries}, we review some preliminary material regarding the Kummer surface of an Abelian surface and in particular the Kummer lattice $K \subset H_2\big(\Km(A)\big)$. We also discuss the basic construction in orbifold Gromov-Witten theory which allows us to study hyperelliptic curves with genus 0 invariants. We provide a partial description of the relevant moduli space in the case that the Picard number of $A$ is 1, and we explain how to obtain enumerative invariants from decidedly non-enumerative ones. 

Section \ref{chap_main} consists of a proof of Theorem \ref{thm_main} obtained by computing a restricted form of the Gromov-Witten potential (see Definition \ref{def_restricted_potential}), followed by applying the crepant resolution conjecture to obtain the corresponding potential function on $[A/\pm1]$. Lastly, we simplify this by accounting for collapsing components to prove Theorem \ref{thm_main}.

Section \ref{sec_low_genera} consists of a proof of the genus one and two case {\em independent of the crepant resolution conjecture}. This involves specializing to the case that $A \cong S \times F$ for generic elliptic curves $S$ and $F$. From there the problem is reduced to counting covers of an elliptic curve, which is classically known.

Lastly, Appendix \ref{appendix_mod} consists of an analysis of the moduli space of genus 0 twisted stable maps into $[A/\pm1]$, as well as a discussion of its reduced virtual fundamental class.


\section{Preliminaries}\label{chap_preliminaries}

\subsection{Abelian surfaces and Kummer surfaces}\label{sec_ab_kum}

For the duration of this paper, unless otherwise noted, all coefficients are integral. The majority of the results in this section follows \cite{bpvdv}, and as such, proofs are omitted.

Let $A$ be an Abelian surface. Then $A$ is a complex torus $\C^2/\Gamma$ with $\Gamma$ of rank 4. As an Abelian group, $A$ has an involution given by multiplication by $\pm1$. This has as fixed points the collection $A[2] \cong \Gamma/2\Gamma$ of sixteen 2-torsion points. The quotient by this action has these as its only singularities, and so by blowing them up we obtain the (crepant) resolution $\Km(A)$ which is a smooth $K3$ surface called the {\em Kummer surface}\index{Kummer surface} of $A$.

If instead we blow up $A$ at the sixteen 2-torsion points to produce $\overline{A}$, we can take the quotient of this space by the lifted involution to obtain the diagram
\begin{equation}\label{kummer_surface_structure}
\xymatrix{
\overline{A}\ar[rr]^\sigma \ar[d]_s & & \Km(A) \ar[d]^p \\
A \ar[rr]_\pi & & A/\pm1
}.
\end{equation}

There are a few facts that this yields, all of which are connected to the affine $\F_2$-geometry of $A[2]$. We begin by introducing some notation.

\begin{enumerate}
\item Let $E_v \in H_2\big(\Km(A)\big)$ denote the class of the $(-2)$-curve lying over a given $v \in A[2]$.
\item Let $\Lambda$ denote the sublattice of $H_2\big(\Km(A)\big)$ generated by the classes $E_v$.
\item Let $K$ denote the minimal primitive sublattice of $H_2\big(\Km(A)\big)$ which contains $\Lambda$. This is called the {\em Kummer lattice}\index{Kummer lattice}.
\end{enumerate}


Furthermore, let $\mathfrak{P}(A[2])$ denote the power set of $A[2]$ (which is a group under the operation $S + S' = S \cup S' \setminus (S \cap S')$) and let $\Pi_k$ for $0 \leq k \leq 4$ denote the subgroups generated by all of the affine $k$-planes in $A[2]$. Then we have that
\[
\Z/2 = \Pi_4 \subset \Pi_3 \subset \Pi_2 \subset \Pi_1 \subset \Pi_0 = \mathfrak{P}(A[2]).
\]

\begin{remark}\label{rmk_correspondence_subset}
We will use throughout this paper the letters $\eta$ and $\varepsilon$ (possibly with subscripts) to denote elements of $\Pi_k$. Note also that for each element $\eta \in \Pi_k$ we can think of $\eta$ as an element of $\tfrac{1}{2}\Lambda$ via the correspondence
\[
\eta \leftrightarrow \hat\eta = \sum_{v\in \eta} \tfrac{1}{2}E_v.
\]
We will also throughout use the notation $|\eta|$ to denote the number of elements in $\eta$.
\end{remark}

\begin{remark}
We should note that there are two notions of summation at play here---summation in $\mathfrak{P}(A[2])$, and summation in $\tfrac{1}{2}\Lambda$. When we write $\eta_1 + \eta_2$ we will always mean the former, and when we write $\hat\eta_1 + \hat\eta_2$ we will always mean the latter, so no confusion should arise.
\end{remark}

As $\Km(A)$ is a smooth real 4-manifold, the group $H_2\big(\Km(A)\big)$ comes with a natural intersection form $\langle\, , \rangle$ which turns it into a unimodular lattice. When restricted to $\Lambda$ this is $(-2)Id$, and so we have that $\Lambda^\vee= \tfrac{1}{2}\Lambda$. Thus we have that
\[
\Lambda \subset K \subset K^\vee \subset \tfrac{1}{2}\Lambda
\]
and so every $w \in K$ (and in $K^\vee$) can be written as
\[
w = \sum_{v \in A[2]}\frac{a_v}{2}E_v.
\]
From this we have a natural map $r : \tfrac{1}{2}\Lambda \to \Pi_0$ given by
\[
\sum_{v \in A[2]} \frac{a_v}{2}E_v \mapsto \big\{v \in A[2] \mid a_v \equiv 1 \pmod 2\big\}.
\]
Note that for $w = \sum_{v \in A[2]}\frac{a_v}{2}E_v$, that $\hat{r(w)}$ is nothing but the reduction of the coefficients $a_v$ of $w$ mod 2.


\begin{remark}\label{rmk_int_form}
It follows further that the intersection form on $K$ can be extended linearly to an intersection form on $\tfrac{1}{2}\Lambda$ which we also denote by $\langle\, , \rangle$. In particular, using the correspondence of Remark \ref{rmk_correspondence_subset}, we can define for any two subsets $\eta_1, \eta_2$ of $A[2]$ the pairing $\langle \hat\eta_1, \hat\eta_2 \rangle$ by
\[
\langle \hat\eta_1, \hat \eta_2 \rangle = \Big\langle \sum_{v \in \eta_1} \frac{1}{2}E_v , \sum_{v \in \eta_2} \frac{1}{2}E_v\Big\rangle = -\tfrac{1}{2}|\eta_1 \cap \eta_2|.
\]
\end{remark}

The condition that $w \in K$ can now be stated as the following.

\begin{proposition}[{\cite[Proposition 5.5, Chapter VIII]{bpvdv}}]
An element $w = \sum_{v\in A[2]} \frac{a_v}{2}E_v \in \tfrac{1}{2}\Lambda$ is in $K$ if and only if $r(w) \in \Pi_3$. 
That is, $w \in K$ if and only if the collection of those $v \in A[2]$ such that $a_v$ is odd is either
\begin{enumerate}
\item empty
\item an affine 3-plane
\item all of $A[2]$.
\end{enumerate}
\end{proposition}

This then yields a description of the Kummer lattice as follows.

\begin{corollary}\label{cor_kum_lat}
There is a short exact sequence
\[
\xymatrix{
0 \ar[r] & \Lambda \ar[r] &K \ar[r]^r & \Pi_3 \ar[r] & 0
}.
\]
\end{corollary}

Following prior remarks, this permits us to consider each element $\eta \in \Pi_3$ as an element $\hat\eta \in K$.

We next describe the relationship between $H_2(A)$ and $H_2\big(\Km(A)\big)$. Consider again the diagram \eqref{kummer_surface_structure}, and in particular, consider the map
\[
\alpha = \sigma_* \circ s^! : H_2(A) \to H_2\big(\Km(A)\big).
\]
We have the following proposition which relates the two intersection forms.

\begin{proposition}[{\cite[Proposition 5.1, Chapter VIII]{bpvdv}}]
The map $\alpha$ multiplies the intersection form by 2. That is, $\alpha(a) \cdot \alpha(b) = 2 a \cdot b$ for every $a, b \in H_2(A)$. Moreover, for each $v \in A[2]$ and for each class $a \in H_2(A)$, $E_v \cdot \alpha(a) = 0$.
\end{proposition}

It follows from this proposition that the map $\alpha$ embeds $H_2(A)$ as a sublattice of $H_2\big(\Km(A)\big)$ which is orthogonal to the Kummer lattice $K$.

We examine now one of the important properties of this map. Recall that $A = \C^2/\Gamma$, let $u : \C^2 \to A$ be the quotient map, and let $\lambda_1, \lambda_2$ be basis elements of $\Gamma$. Define $V = u\big(\langle \lambda_1, \lambda_2\rangle \otimes \R\big)$ (the image of the 2-plane spanned by $\lambda_1, \lambda_2$), and let $C = [V]$. Then $\alpha(C)$ is the class of the proper transform of $\pi(V + t)$ in $\Km(A)$ for some generic $t \in A$ (i.e. such that $V + t$ does not intersect $A[2]$).

Next, consider the rational curve $V/\pm1 \subset A/\pm1$, and let $\beta \in H_2\big(\Km(A)\big)$ be the class of the proper transform of $V/\pm1$. Let $\varepsilon$ denote the collection of 2-torsion points in $V$. We have the following relation between the classes $\alpha(C)$ and $\beta$.

\begin{proposition}\label{prop_congruence}
Let $V$, $C$, and $\beta$ be as above. The classes $\alpha(C)$ and $\beta$ satisfy the relationship
\begin{align*}
\beta &= \frac{1}{2}\alpha(C) - \hat\varepsilon \\
 &= \frac{1}{2}\alpha(C) - \frac{1}{2}\sum_{v \in \varepsilon} E_v
\end{align*}
\end{proposition}

\begin{proof}
As $\alpha(C)$ is the class of the proper transform in $\Km(A)$ of a generic translate of $V$, we see that in $\Km(A)$ we must have that
\[
\beta = \frac{1}{2}\alpha(C) - \frac{1}{2}\sum_{v \in A[2]} a_vE_v
\]
for some integers $a_v$. Since $\beta \cdot E_v = a_v$,  we see that $a_v = 1$ if $v \in \varepsilon$ and is zero otherwise, as claimed.
\end{proof}

Recall that there is a canonical isomorphism $H^2(A) \cong \bigwedge^2 H_1(A)^\vee$, and so we can regard elements of $H^2(A)$ as alternating forms on $H_1(A)$. Suppose now that $(A_n,L_n)$ is a polarized abelian surface with polarization of type $(1, n)$.  That is, there is a basis $e_1, f_1, e_2, f_2$ of $H_1(A_n)$ so that $c_1(L_n)$ (when viewed as an alternating form) can be written as
\begin{equation}\label{eq_polarized_basis}
\begin{pmatrix}
0 & 0 & n & 0 \\
0 & 0& 0 & 1\\
-n & 0 & 0 & 0 \\
0 & -1 & 0 & 0
\end{pmatrix}
\end{equation}
In this basis, we can write $c_1(L_n)^\vee = (e_1 \wedge f_1)  + n (e_2 \wedge f_2)$.

\subsection{Gromov-Witten theory of hyperelliptic curves}\label{sec_gw_hyper}

We aim to compute the number of hyperelliptic curves in $A$ via orbifold Gromov-Witten theory following the ideas of \cite{wise, graber,gillam}. Let $\mathscr{X}$ be a smooth Deligne-Mumford stack with projective coarse moduli space $X$. We will use the notation
\[
\moda[\mathscr{X}]{2g+2}{\beta}
\]
to denote the moduli stack of twisted stable maps of genus 0 curves into the stack $\mathscr{X}$ in the curve class $\beta \in H_2(X)$ with $(2g+2)$ $\Z/2$-stacky points. In the case where $A$ is a polarized Abelian surface with polarization of type $(1, n)$, we will use the notation
\[
\moda{2g+2}{n}
\]
to denote the moduli stack where the class $\beta$ is the class $\tfrac{1}{2}\pi_*c_1(L_n)^\vee$.


As with ordinary Gromov-Witten theory, there are evaluation maps from the moduli stack of twisted stable maps. However, they do not lie in $\mathscr{X}$, but in its {\em rigidified inertia stack}, $I\mathscr{X}$ (see \cite{agv, agv_long}). In the case that $\mathscr{X} = [X/G]$ is a global quotient, this has a particularly simple description.

\begin{definition}\label{def_inertia_stack}
Let $\mathscr{X} = [X/G]$ be a global quotient stack. We define the {\em rigidified inertia stack} to be
\[
I\mathscr{X} = \coprod_{(g) \subset G} [X^g/H_g]
\]
where the disjoint union is taken over all conjugacy classes $(g) \subset G$, where $X^g$ is the fixed-point set of $g$, and where $H_g = C(g)/\langle g\rangle$ is the quotient of the centralizer of $g$ in $G$ by the subgroup generated by $g$. The component corresponding to $(e) \subset G$ is called the {\em untwisted sector}, while all others are called {\em twisted sectors}.
\end{definition}

For the case of the the quotient stack $[A/\pm1]$, it is easy to see that 
\[
I[A/\pm1] = A \amalg A[2]
\]
and so the twisted sector is identified with $A[2]$, the set of 2-torsion points of $A$.


As before, let $(A,L)$ be a polarized Abelian surface. We have a map
\[
\xymatrix{
A \ar[r] \ar@/_1.5pc/[rr]_{\Delta} & A \times A \ar[r] & Sym^2A
}
\]
where the first map is given by $a \mapsto (a,-a)$. Next, we consider the moduli space
\[
\moda[{[Sym^2A]}]{2g+2}{\tfrac{1}{2}\Delta_*c_1(L)^\vee}
\]
(where the factor of $\tfrac{1}{2}$ comes from the fact that a curve in $A$ is a double cover of the corresponding curve in $Sym^2A$). This parameterizes genus 0, twisted stable maps into the stack $[Sym^2A]$. As these maps are representable, we may complete them to a diagram
\[
\xymatrix{
\tilde{C} \ar[d]\ar[rr] && A \times A \ar[d] \\
C \ar[rr] && [Sym^2A]
}
\]
with $\tilde{C}$ a scheme and the top map equivariant. If $C$ is smooth, then $\tilde{C}$ is a smooth hyperelliptic curve and the projection onto either factor yields a hyperelliptic curve in $A$ in the class $c_1(L)^\vee$. It follows that the moduli space $\moda[{[Sym^2A]}]{2g+2}{\tfrac{1}{2}\Delta_*c_1(L)^\vee}$ is a compactification of the moduli space of smooth hyperelliptic curves in $A$.

So far this follows very closely \cite{wise, gillam}. In our case, we can simplify matters significantly. As above, we look at the diagram
\[
\xymatrix{
A \ar[rr]^{a\mapsto (a, -a)}\ar[d]_\pi && A \times A \ar[d] \\
[A/\pm1] \ar[rr]_\iota && [Sym^2A]
}
\]
where the map $\iota : [A/\pm1] \to [Sym^2A]$ is given by $[a] \mapsto [a,-a]$.

\begin{theorem}
We have the following isomorphism of stacks.
\[
\moda{2g+2}{\beta} \times A \ \cong \ \moda[{[Sym^2A]}]{2g+2}{\iota_*\beta}
\]
\end{theorem}

\begin{proof}
The map in one direction is easy to produce. A family of objects in $\moda{2g+2}{\beta}$ over a base scheme $B$ consists of a diagram
\[
\xymatrix{
\tilde{C} \ar[r]\ar[d] & A \ar[d] \\
C \ar[r]\ar[d] & A/\pm1 \\
B \ar@/^1pc/[uu]^{s_i}
}
\]
with $C$ of genus 0 and with $(2g+2)$ sections $s_i$ (thus if $C$ is smooth, $\tilde{C}$ is hyperelliptic of genus $g$). Given an element $a_0$ of $A$ we can construct a map $A \to A \times A$ given by
\[
a \mapsto (\tfrac{1}{2}a_0 + a, \tfrac{1}{2}a_0 - a)
\]
which we then complete to
\[
\xymatrix{
\tilde{C} \ar[r]\ar[d] & A \ar[d] \ar@{.>}[r] & A \times A  \ar[d]\\
C \ar[r]\ar[d] & A/\pm1 \ar@{.>}[r] & Sym^2A\\
B \ar@/^1pc/[uu]^{s_i}
}
\]
which yields the first half.

For the second half, note that there is a map $+ : Sym^2A \to A$ given by $[a,b] \mapsto a + b$. Thus given the diagram
\[
\xymatrix{
\tilde{C} \ar[r]\ar[d] & A \times A \ar[d] \\
C \ar[r]\ar[d] & Sym^2A \ar@{.>}[r]^<<<<<+ & A\\
B \ar@/^1pc/[uu]^{s_i}
}
\]
with $C$ a rational curve. Since there are no rational curves in Abelian surfaces, the composition $C \to Sym^2A \to A$ must be constant, and so the diagram factors through the inclusion of a fibre of the map to $A$. As these are all isomorphic to $A/\pm1$, the claim follows.
\end{proof}

We can interpret this theorem as saying that counting hyperelliptic curves in $A$ is equivalent to counting certain stacky rational curves in the orbifold $[A/\pm1]$. Using the crepant resolution conjecture, this should be the same as counting certain rational curves in the smooth $K3$ surface $\Km(A)$. This has been studied in \cite{yau_zaslow, bryan_leung_k3}.




In \cite{bryan_leung_abelian}, it is shown that the (reduced) Gromov-Witten invariants for an Abelian surface only depend on the divisibility and square of the class $\beta$. This follows because of the fact that the moduli space $\mathcal{A}_{2n}$ of polarized Abelian surfaces whose polarizations have square $2n$ is connected, and from the deformation invariance of Gromov-Witten invariants.

We can use this fact to show that the same holds true in our case; indeed, there is a surjective map from $\mathcal{A}_{2n}$ to the moduli space of singular Kummer surfaces with polarizations with square $2n$ (given by taking the quotient by $\pm1$), and so the same result holds. This justifies our notation of $\moda{2g+2}{n}$ (specifically its lack of dependence on the class $\beta$).

As twisted stable maps are representable, each of the $(2g+2)$ evaluation maps $ev_i : \moda{2g+2}{n} \to I[A/\pm1]$ must all lie in the twisted sector, which is $A[2]$. Let $ev$ denote the map 
\[
ev : \moda{2g+2}{n} \to A[2]^{2g+2}
\]
and let $v_1, \ldots, v_{16}$ denote an arbitrary labeling of $A[2]$.

Given $\bk : A[2] \to \zp$ with $|\bk| = 2g+2$, we define
\[
\moda{\bk}{n} = ev^{-1}(\underbrace{v_1, \ldots, v_1}_{\bk(v_1)}, \ldots, \underbrace{v_{16}, \ldots, v_{16}}_{\bk(16)}).
\]
This is the moduli space of those orbifold maps with $\bk(v)$ stacky points whose image lies on each $v \in A[2]$. This space has a degree 0 reduced virtual fundamental class (see Appendix \ref{appendix_mod}), and so  we define
\[
GW_{\bk,n} = \deg\,\big[\moda{\bk}{n}\big]^{red}.
\]
It should be noted that in the definition of $GW_{\bk,n}$, the specific labelling of the element in $A[2]^{|\bk|}$ is not relevant;  any rearrangement of those terms comes simply from a permutation of the labelling of the marked points, and will yield the same Gromov-Witten invariant.

We would like to use this invariant to count hyperelliptic curves, but it is not enumerative; it includes contributions from collapsing components. In the case that $A$ is suitably generic (i.e. has Picard number 1), we can use this to determine the enumerative count as follows (cf. \cite{wise,graber}).

\begin{definition}\label{def_comb}
An $n$-marked {\em comb curve} is a genus 0 twisted stable map $f : \Sigma \to [A/\pm1]$ with 0 ordinary marked points and $n$ $\Z/2$-marked points such that there is a unique irreducible component $\Sigma_0$ which has nonzero degree; this component is called the {\em handle}. All other components are called {\em teeth}.
\end{definition}

It is clear that the locus of $(2g+2)$-marked comb curves $U_{2g+2,n}$ is a closed substack of $\moda{2g+2}{n}$. Similarly, define $U_{\bk,n} = U_{|\bk|,n} \cap \moda{\bk}{n}$. In the case that the Picard number of $A$ is 1, it turns out that all curves are comb curves, as the following proposition shows.

\begin{proposition}\label{prop_comb_curve_reduction}
If the Picard number of $A$ is 1, then the moduli spaces $U_{2g+2,n}$ and $\moda{2g+2}{n}$ are equal, and similarly for $U_{\bk, n}$.
\end{proposition}

\begin{proof}
Suppose that there were more than one component with non-zero degree. Since the class in $[A/\pm1]$ is primitive, this would immediately yield that the Picard number is greater than 1.
\end{proof}

\begin{figure}[ht]
\begin{center}
\includegraphics{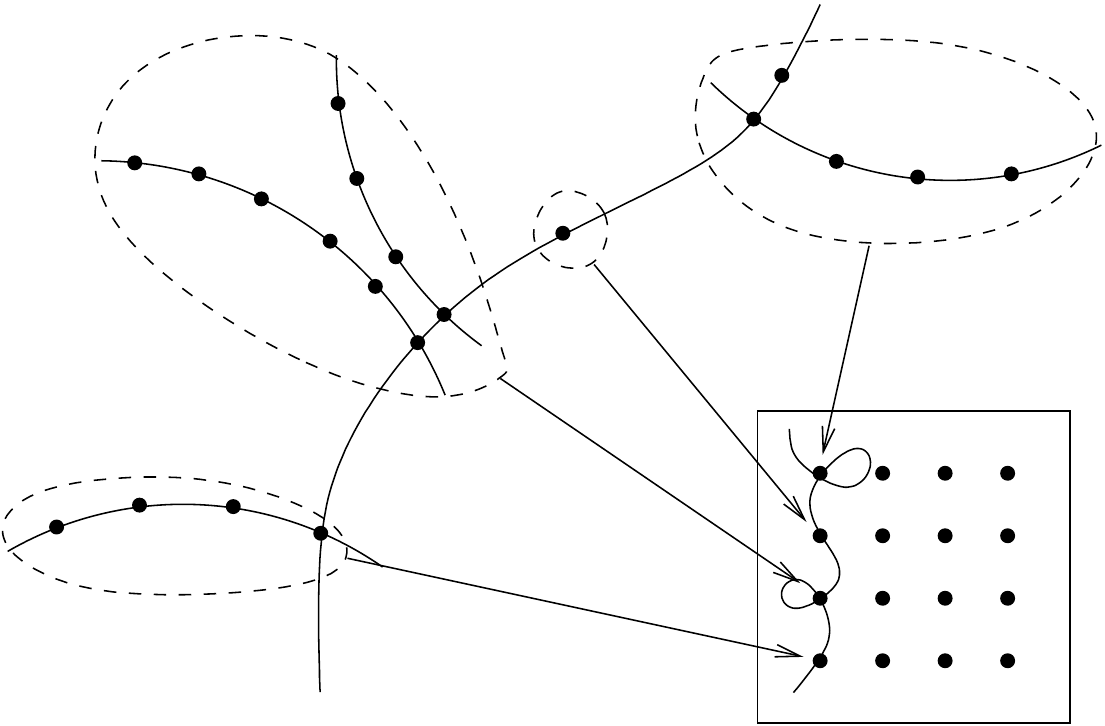}
\end{center}
\caption[A map in a component of $\moda{\bk}{n}$]{A map in the component of $\moda{\bk}{n}$ corresponding to the partitions $\lambda = (\lambda^1, \lambda^2, \lambda^3, \lambda^4) = \big([3],[3\, 5],[1],[1\, 3]\big)$}\label{fig_tw_map}
\end{figure}

As discussed in Appendix \ref{appendix_mod}, this moduli space splits up into components based on how the marked points partition among the teeth (see figure \ref{fig_tw_map}). From Lemma \ref{lem_even_parts}, we see that all partitions with even parts contribute zero to the Gromov-Witten invariant, and so for each comb curve $\Sigma$ and each $v \in A[2]$ we obtain a partition $\lambda^v = (\lambda_1^v, \ldots, \lambda_{r_v}^v)$ of $\bk(v)$ (where $r_v$ denotes the length of the partition $\lambda^v$) into odd parts based on how the marked points are split up among the teeth of $\Sigma$ (with $\lambda_i^v = 1$ being interpreted as there being no collapsing component---that is, it represents a stacky marked point on the handle). Let  $U_{\lambda, n}$ denote the component consisting of those comb curves with partition type $\lambda = (\lambda^v)_{v \in A[2]}$. If we define $\bk_\lambda(v) = r_v$, then since the smoothing of any node is obstructed (see Proposition \ref{prop_disjoint_union}) and since all collapsing components must have image a stacky point in $[A/\pm1]$, it is clear that
\[
U_{\lambda,n} = \modac{\bk_\lambda}{n} \times \prod_{v\in A[2]} \prod_{i=1}^{r_v} \moda[B\Z/2]{\lambda_i^v+1}{0}
\]
where $\modac{\bk}{n}$ denotes the component consisting of those curves with no collapsing components.

Let $p_\lambda$ denote the projection $U_{\lambda,n} \to \modac{\bk_\lambda}{n}$.
We have the following theorem, whose proof we defer to Appendix \ref{appendix_mod}.

\begin{theorem}\label{thm_virt_degree}
Let $\lambda = (\lambda^v)_{v\in A[2]}$ be a collection of partitions of $\bk$, all of which consist of odd parts. Then the virtual degree of $p_\lambda$ is $\big(-\frac{1}{4}\big)^{\frac{1}{2}(|\bk| - |\bk_\lambda|)}$. That is,
\[
(p_\lambda)_*[U_{\lambda,n}]^{red} = \Big(-\frac{1}{4}\Big)^{\frac{1}{2}(|\bk| - |\bk_\lambda|)}\big[\modac{\bk_\lambda}{n}\big]^{red}.
\]
\end{theorem}

We note that $U_{\bk,n}$ is the disjoint union of the $U_{\lambda,n}$ taken over all partition types $\lambda$. In particular, there is a component $M_\bk^\circ$ which consists of those curves with no collapsing components (corresponding to the partitions $1^{\bk(v)}$). We can now define
\begin{equation}\label{gwcirc}
GW_{\bk,n}^\circ = \deg [M_\bk^\circ]^{red}.
\end{equation}
From this number we obtain our expected count of hyperelliptic curves. More precisely, in the generic setting we expect that this counts the number of hyperelliptic curves together with the extra data of an ordering on the marked points which collapse to a given 2-torsion point. As such, define
\begin{equation}\label{nkn}
N_{\bk,n} = \frac{GW_{\bk,n}^\circ}{\prod_{v \in A[2]}\bk(v)!}
\end{equation}


Our main theorem of this section describes this relationship. Define the generating functions
\begin{gather*}
F_n(z_v) = \sum_{\bk : A[2] \to \zp} GW_{\bk,n} \prod_{v \in A[2]} \frac{z_v^{\bk(v)}}{\bk(v)!} \\
F_n^\circ(x_v) = \sum_{\bk : A[2] \to \zp} GW_{\bk,n}^\circ \prod_{v \in A[2]} \frac{x_v^{\bk(v)}}{\bk(v)!}
\end{gather*}

\begin{theorem}\label{thm_substitution}
The two generating functions $F_n$ and $F_n^\circ$ are equal after the substitution $x_v = 2 \sin(z_v/2)$.
\end{theorem}

A proof of this theorem will be provided in Appendix \ref{appendix_mod}. However, from the equation \eqref{nkn}, we immediately obtain the following.

\begin{corollary}
Let $\bk : A[2] \to \Z_{\geq 0}$ be a function and let $g$ be such that $|\bk| = 2g + 2$. Then the function $F_{g,\bk}$ of theorem \ref{thm_main} is given by  the coefficient of $\prod_{v\in A[2]}x_v^{\bk(v)}$ in
\begin{align*}
\sum_{n = 0}^\infty F_n^\circ(x_v)u^n &= \sum_{n=0}^\infty \sum_{\bk : A[2] \to \Z_{\geq 0}} GW_{\bk,n}^\circ \prod_{v \in A[2]} \frac{x_v^{\bk(v)}}{\bk(v)!}u^n \\
 &= \sum_{n=0}^\infty \sum_{\bk : A[2] \to \Z_{\geq 0}} N_{\bk,n} \prod_{v \in A[2]} x_v^{\bk(v)}u^n.
\end{align*}
\end{corollary}

\begin{remark}
In the case that the Picard number of $A$ is 1, we can define $GW_{\bk, n}^\circ$ and $N_{\bk, n}$ directly via equations \eqref{gwcirc} and \eqref{nkn}. However, they are defined for all Abelian surfaces $A$ via the relationships given  by Theorem \ref{thm_substitution}.
\end{remark}

%
%

\subsection{The case $A \cong S \times F$}\label{sec_specialization}

We will sometimes have need to specialize to the case that $A \cong S \times F$ for generic elliptic curves $S$ and $F$. In such a case, the quotient $A/\pm1$ comes equipped with an elliptic fibration $A/\pm1 \to S/\pm1$ whose general fibre is the elliptic curve $F$. This has four distinguished fibres over $S[2]$ which are isomorphic to $F/\pm1$.

So let $f : \Sigma \to [A/\pm1]$ be a genus 0 twisted stable map in the class $[S/\pm1] + n [F/\pm1]$ (which we will assume without loss of generality collapses no components). In such a case, the source must be a tree of rational curves $\Sigma_0 \cup \Sigma_1 \cup \cdots \cup \Sigma_k$. We can label these so that $f_*[\Sigma_0] = [S/\pm1]$ and $f_*[\Sigma_i] = n_i[F/\pm1]$ for $i \geq 1$. In particular, $f |_{\Sigma_0}$ is an isomorphism (and so must contain exactly 4 stacky points), while $f |_{\Sigma_i}$ is a ramified cover of one of the distinguished fibres, i.e. the rational orbi-curve $[F/\pm1]$.

As a (representable) ramified cover of an orbi-curve, it follows that
\begin{enumerate}
\item The image of each stacky point is stacky.
\item The pre-image of each stacky point is a collection of stacky points (with odd ramification) and some non-stacky points (with even ramification).
\end{enumerate}

In such a case, the moduli space $\moda{2g+2}{n}$ will be isomorphic to a product of spaces of Hurwitz covers of the fibre $F/\pm1$, a fact which will be necessary in order to compute the orbifold Gromov-Witten theory of $[A/\pm1]$ {\em without} using the crepant resolution conjecture.


\section{Main Work}\label{chap_main}

\subsection{Computation of the invariants on $\Km(A_n)$}

The goal of this section is to compute the relevant part of the Gromov-Witten potential function for $\Km(A_n)$.

Let $f : \Sigma \to A_n$ be a hyperelliptic curve representing the class $c_1(L_n)^\vee$ and such that $f(w) \in A[2]$ for all Weirstrass points $w$. Let $\beta_n$ be the class in $H_2\big(\Km(A_n)\big)$ of the proper transform of $(\pi \circ f)(\Sigma)$; note that this is a rational curve. Let $H_n = \alpha\big(c_1(L_n)^\vee\big)$. If we choose a basis of $H_1(A_n)$ as in section \ref{sec_ab_kum}, then we can write $H_n = S + nF$ where $S = \alpha(e_1 \wedge f_1)$ and $F = \alpha(e_2 \wedge f_2)$.

Recall now that, given two basis elements $\lambda_1, \lambda_2$ of $\Lambda$, we can consider the image $V$ in $A$ of the real 2-plane spanned by $\lambda_1, \lambda_2$. We can further consider the homology class $\beta_V$ of the proper transform of $V/\pm1$ in $\Km(A)$, which by Proposition \ref{prop_congruence} can be written as
\[
\beta_V = \frac{1}{2}\alpha\big([V]\big) - \frac{1}{2}\sum_{v\in\varepsilon_V}E_v.
\]
We see that to each pair of basis elements we can associate en element $\varepsilon_V \in \Pi_0$ (since $\lambda_1, \lambda_2$ span a plane, this will in fact be in $\Pi_2$).

Let $\varepsilon_0$ be the element so obtained from the basis elements $e_1, f_1$, let $\varepsilon_1^*$ be the element so obtained from the basis elements $e_2, f_2$, and let finally $\varepsilon_1 = \varepsilon_0 + \varepsilon_1^*$ (where, recall, the summation is done in $\mathfrak{P}(A[2])$ and is hence the symmetric difference of the two elements).

\begin{proposition}\label{prop_cong}
We have the following congruences.
\begin{gather*}
\beta_{2k} \equiv \frac{1}{2}H_{2k} - \hat\varepsilon_0 \pmod K \\
\beta_{2k+1} \equiv \frac{1}{2}H_{2k+1} - \hat\varepsilon_1 \pmod K
\end{gather*}
\end{proposition}

\begin{proof}
We assume as per Section \ref{sec_gw_hyper} that the image curve is fixed by $\pm1$ and so descends to a genus 0 map $\tilde{f} : \Sigma \to [A/\pm1]$. We assume further (by deformation invariance) that $A \cong S \times F$, which puts us in the situation described in Section \ref{sec_specialization}. We will further assume that $n$ is even, the odd case being similar.

Let $\Sigma = \Sigma_0 \cup \Sigma_1 \cup \cdots \cup \Sigma_k$ be the source curve. Since $n$ is even, we must have that either the degree of each $\Sigma_i$ is even, or that any odd ones come in pairs. As in the proof of Proposition \ref{prop_congruence}, we have that $H_n$ is the class of a double cover of the proper transform of the image of $\tilde f$ in $\Km(A)$, and so we have that
\[
\beta_n = \frac{1}{2}H_n - \frac{1}{2}\sum_{v\in A[2]} a_vE_v
\]
where $a_v$ is the intersection multiplicity of the proper transform with the exceptional curve $E_v$. This is given by the sum of all the ramification indices over all points $p \in \Sigma$ which map to $v \in A[2]$. Since $E_v \in K$, this only depends (mod $K$) on the ramification indices mod 2. Since $\tilde f_* [\Sigma_i] = n_i[F/\pm1]$, it follows that the image of all the stacky points on a given $\Sigma_i$ lie in $\varepsilon_1^*$, or some translation thereof.

Recall from Section \ref{sec_specialization} that the ramification must be even at each non-stacky point in the pre-image, and is odd at each stacky point in the pre-image, and hence $a_v \equiv \bk(v) \pmod 2$.

Consider first a component $\Sigma_i$ with $\tilde f|_{\Sigma_i}$ of even degree. Due to these ramification considerations, it follows that over each $v \in A[2]$ the number of stacky pre-images must be even, and so (mod $K$) this contributes zero to $\beta_n$.

Consider next components $\Sigma_i, \Sigma_j$ with $\tilde f|_{\Sigma_i}$ and $\tilde f|_{\Sigma_j}$ of odd degree (recall that these must come in pairs). We see that for both of these the number of stacky pre-images of a given $v$ must be odd, and so one of two things occur.
\begin{enumerate}
\item $\Sigma_i$ and $\Sigma_j$ map via $\tilde f$ to the same fibre.
\item $\Sigma_i$ and $\Sigma_j$ map to different fibres.
\end{enumerate}
In the first case, the contribution to the number of stacky pre-images winds up being even, and so (mod $K$) contributes zero. In the second case, the stacky points form an affine 3-plane $\eta$ in $A[2]$---but since $\hat\eta \in K$, it follows that (mod $K$) these also contribute nothing. As such, all that contributes (mod $K$) are the stacky points coming from the curve $\Sigma_0$. But this is exactly $\varepsilon_0$, which completes the proof.

\end{proof}

\begin{remark}\label{rmk_condition}
The condition on $P$ so that the generating function $F_{g,\bk}(u)$ is non-zero can be now explained as follows. Recall that $P$ is the collection of those $v \in A[2]$ such that $\bk(v)$ is odd. It will arise naturally during the proof of Theorem \ref{thm_main} that we must have $P \equiv \varepsilon_i \pmod K$. One consequence of this is that if $P \equiv \varepsilon_0 \pmod K$, then only even polarizations can occur (and conversely for $P \equiv \varepsilon_1 \pmod K$).
\end{remark}

We are now ready to compute the potential function. We first recall the definition of the Gromov-Witten potential of a smooth projective variety.

\begin{definition}\label{def_gw_potential}
Let $X$ be a smooth projective variety, and let $\gamma_0, \ldots, \gamma_a$ be an additive basis of $H^*(X)$. The genus 0 {\em Gromov-Witten potential function} is defined by
\[
F^X(y_0, \ldots, y_a, q) = \sum_{m_0, \ldots, m_a = 0}^\infty \sum_{\beta \in H_2(X)}\langle \gamma_0^{m_0} \cdots \gamma_a^{m_a}\rangle_\beta^X \frac{y_0^{m_0}}{m_0!} \cdots \frac{y_a^{m_a}}{m_a!} q^\beta.
\]
\end{definition}

As stated in the introduction, we are only concerned with a restricted form of this function. More specifically, we are only concerned with homology classes $\beta$ such that $p_*\beta = \tfrac{1}{2}\pi_*c_1(L_n)^\vee$, and so we make the following definition.

\begin{definition}\label{def_restricted_potential}
Define the {\em restricted genus 0 potential function} as
\begin{multline*}
F_n := F^{\Km(A_n)}(y_0, \ldots, y_a, q) = \\
\sum_{m_0, \ldots, m_a = 0}^\infty \sum_{w \in K}\langle \gamma_0^{m_0} \cdots \gamma_a^{m_a}\rangle_{\beta_n+w}^{\Km(A_n)}\frac{y_0^{m_0}}{m_0!} \cdots \frac{y_a^{m_a}}{m_a!}  q^{\beta_n+w}.
\end{multline*}
\end{definition}

\begin{remark}\label{rmk_basis}
Due to dimension considerations, the only classes that will produce non-zero invariants will be divisor classes. As such, we choose as a basis of $H^2\big(\Km(A_n)\big)$ the classes $\gamma_v$ which are dual to the exceptional curve classes $E_v$, as well as classes $\gamma_S, \gamma_F$ dual to $S$ and $F$, respectively. As such, if we define $\mathbf{m} = (m_{v_1}, \ldots, m_{v_{16}}, m_S, m_F)$ then we can write the function $F_n$ as
\begin{multline*}
F_n(y_{v_1}, \ldots, y_{v_{16}}, y_S, y_F,q) = \\
\sum_{\mathbf{m}} \sum_{w \in K}\langle \gamma_S^{m_S}\gamma_F^{m_F} \prod_{v\in A[2]} \gamma_v^{m_v}\rangle_{\beta_n+w}\frac{y_S^{m_S}}{m_S!} \frac{y_F^{m_F}}{m_F!}\prod_{v \in A[2]}\frac{y_v^{m_v}}{m_v!} q^{\beta_n+w}
\end{multline*}
where we omit for simplicity of notation the superscript on the brackets $\langle \cdots \rangle_\beta^{\Km(A_n)}$.

Moreover, define the generating function
\[
F := \sum_{n=0}^\infty F_n(y_{v_1}, \ldots, y_{v_{16}},y_S,y_F,q).
\]
{\em A priori} this does not make sense, but as we will see, the formal variable $q$ used to define $F_n$ permits this to be well defined as a formal power series in the variables $y_\alpha$ and $q$. It is this function that we will use (with the crepant resolution conjecture) to compute the number of hyperelliptic curves in $A$.
\end{remark}

\begin{theorem}\label{thm_restricted_potential_km}
The restricted genus 0 potential function $F$ is given by
\begin{align*}
F &= \lambda_0 \frac{u^2}{\Delta(u^2)}\sum_{w \in K} u^{-\langle w, w - 2\hat\varepsilon_0\rangle} \prod_{v \in A[2]}\exp \big((\textstyle\int_w\gamma_v)y_v\big)q^w \\
 &\quad + \lambda_1 \frac{u^2}{\Delta(u^2)}\sum_{w \in K} u^{-\langle w, w - 2\hat\varepsilon_1\rangle+1} \prod_{v \in A[2]}\exp \big((\textstyle\int_w\gamma_v)y_v\big)q^w
\end{align*}
where $\langle\, , \rangle$ is the intersection form on $K \subset \Km(A_n)$, where $u$ and $\lambda_i$ are given by
\begin{gather}\label{eq_variables}
u = \big(\exp(y_F)q^F\big)^{1/2} \notag\\
\lambda_i = \big(\exp(y_S)q^S\big)^{1/2}\prod_{v \in A[2]}\exp\big((\textstyle\int_{-\hat\varepsilon_i}\gamma_v)y_v\big)q^{-\hat\varepsilon_i}
\end{gather}
and where $\Delta(q)$ the weight 12 cusp form defined as $\Delta(q) = q\prod_{k=1}^\infty (1 - q^k)^{24}$ which satisfies
\[
\frac{q}{\Delta(q)} = 1 + 24q + 324q^2 + 3200q^3 + \cdots
\]

\end{theorem}

\begin{proof}
Let $(A_n, L_n)$ be a polarized abelian surface with polarization type $(1, n)$ and with Kummer surface $\Km(A_n)$. As stated in Definition \ref{def_restricted_potential}, the (restricted) Gromov-Witten potential function for this is
\[
F_n = \sum_{\mathbf{m}} \sum_{w\in K} \big\langle \gamma_S^{m_S}\gamma_F^{m_F} \prod_{v\in A[2]} \gamma_v^{m_v}\big\rangle_{\beta_n + w} \frac{y_S^{m_S}}{m_S!} \frac{y_F^{m_F}}{m_F!}\prod_{v \in A[2]}\frac{y_v^{m_v}}{m_v!} q^{\beta_n + w}
\]
Since on $\Km(A_n)$, as stated above, we only need to consider divisor classes the divisor equation simplifies this to
\begin{align*}
F_n &= \sum_{\mathbf{m}} \sum_{w\in K} \langle\, \rangle_{\beta_n + w} \frac{(\int_{\beta_n+w}\gamma_S)^{m_S}y_S^{m_S}}{m_S!} \frac{(\int_{\beta_n+w}\gamma_F)^{m_F}y_F^{m_F}}{m_F!}\cdot \\
&\qquad\qquad\qquad\qquad\prod_{v \in A[2]} \frac{(\int_{\beta_n+w}\gamma_v)^{m_v}y_v^{m_v}}{m_v!} q^{\beta_n + w} \\
 &= \sum_{w \in K} \langle \, \rangle_{\beta_n + w} \exp\big((\textstyle\int_{\beta_n+w}\gamma_S)y_s\big)\exp\big((\textstyle\int_{\beta_n+w}\gamma_F)y_F\big) \cdot \\
 &\qquad\qquad\qquad\qquad\prod_{v \in A[2]} \exp\big((\textstyle\int_{\beta_n+w}\gamma_v) y_v\big) q^{\beta_n + w}.
\end{align*}

Since we are summing over $K$, by proposition \ref{prop_cong} we can replace $\beta_n + w$ with $\tfrac{1}{2} H_n + w - \hat\varepsilon_i$ (with $i \equiv n \pmod 2$), and so noting that (see remark \ref{rmk_basis})
\begin{align*}
\int_{\tfrac{1}{2}H_n + w - \hat\varepsilon_i}\gamma_S &= \tfrac{1}{2} \\
\int_{\tfrac{1}{2}H_n + w - \hat\varepsilon_i}\gamma_F &= \tfrac{n}{2} \\
\int_{\tfrac{1}{2}H_n + w - \hat\varepsilon_i}\gamma_v &= \int_{w - \hat\varepsilon_i}\gamma_v
\end{align*}
we can write this as
\begin{multline*}
F_{2k} = \exp(\tfrac{1}{2}y_S)\sum_{w \in K} \langle\, \rangle_{\frac{1}{2}H_{2k} + w-\hat\varepsilon_0} \exp(\tfrac{1}{2}y_F)^{2k} \\
\prod_{v \in A[2]} \exp\big((\textstyle\int_{w - \hat\varepsilon_0}\gamma_v)y_v\big) q^{\tfrac{1}{2}H_{2k} + w-\hat\varepsilon_0}
\end{multline*}
(and similarly for $F_{2k+1}$).

Let $C_d$ be a primitive curve class in a $K3$ surface $X$ satisfying $C_d^2 = 2d - 2$, and let $N_d = \langle\, \rangle_{C_d}^X$ be the (reduced) genus 0 Gromov-Witten invariant. Then the Yau-Zaslow theorem \cite{yau_zaslow, bryan_leung_k3} states that these numbers satisfy
\[
\sum_{d = 0}^\infty N_dq^d = \frac{q}{\Delta(q)} 
\]
where $\Delta(q)$ is defined above to be $\prod_{k=1}^\infty (1 - q^k)^{24}$. Since $H_n \in \Lambda^\perp$, and $H_n^2 = 4n$, we have that
\[
\big(\tfrac{1}{2}H_{2k} + w - \hat\varepsilon_0\big)^2 = 2k - 2 + \langle w, w - 2\hat\varepsilon_0\rangle
\]
(and similarly for the odd case). Thus we find
\begin{gather*}
\langle\, \rangle_{\frac{1}{2}H_{2k} + w-\hat\varepsilon_0} = N_{k + \tfrac{1}{2}\langle w, w - 2\hat\varepsilon_0\rangle}\\
\langle\, \rangle_{\frac{1}{2}H_{2k+1} + w-\hat\varepsilon_0} = N_{k + \tfrac{1}{2}\langle w, w - 2\hat\varepsilon_1\rangle}
\end{gather*}

As mentioned in Remark \ref{rmk_basis}, let $F = \sum_{n=0}^\infty F_n$. We have
\begin{align*}
F &= \sum_{k=0}^\infty F_{2k} + \sum_{k=0}^\infty F_{2k+1} \\
 &= \exp(\tfrac{1}{2}y_S)\sum_{k=0}^\infty \sum_{w \in K} N_{k + \tfrac{1}{2}\langle w , w - 2\hat\varepsilon_0 \rangle} \exp(\tfrac{1}{2}y_F)^{2k} \\
 & \qquad\qquad\qquad \prod_{v \in A[2]} \exp\big((\textstyle\int_{w - \hat\varepsilon_0}\gamma_v)y_v\big) q^{\tfrac{1}{2}H_{2k} + w-\hat\varepsilon_0} +\\
 & \quad \exp(\tfrac{1}{2}y_S)\sum_{k=0}^\infty\sum_{w \in K} N_{k + \tfrac{1}{2}\langle w , w - 2\hat\varepsilon_1 \rangle}\exp(\tfrac{1}{2}y_F)^{2k+1} \\
 & \qquad\qquad\qquad \prod_{v \in A[2]} \exp\big((\textstyle\int_{w - \hat\varepsilon_1}\gamma_v)y_v\big) q^{\tfrac{1}{2}H_{2k+1} + w-\hat\varepsilon_1}.
\end{align*}
If we then perform the substitutions given in equation \eqref{eq_variables}, this simplifies to
\begin{align*}
F &= \lambda_0\sum_{k=0}^\infty \sum_{w \in K} N_{k + \tfrac{1}{2}\langle w , w - 2\hat\varepsilon_0 \rangle} u^{2k} \prod_{v \in A[2]}\exp\big((\textstyle\int_w\gamma_v)y_v\big)q^w \\
 &\quad \lambda_1\sum_{k=0}^\infty \sum_{w \in K} N_{k + \tfrac{1}{2}\langle w , w - 2\hat\varepsilon_1 \rangle} u^{2k+1} \prod_{v\in A[2]}\exp\big((\textstyle\int_w\gamma_v)y_v\big)q^w
\end{align*}

To simplify this we perform the substitution $n = k + \tfrac{1}{2}\langle w, w - \hat\varepsilon_i\rangle$ which yields
\begin{align*}
F &= \lambda_0 \sum_{n=0}^\infty \sum_{w \in K} N_n u^{2n}u^{-\langle w, w - 2\hat\varepsilon_0\rangle} \prod_{v\in A[2]}\exp\big((\textstyle\int_w\gamma_v)y_v\big)q^w \\
 &\quad + \lambda_0 \sum_{n=0}^\infty \sum_{w \in K} N_n u^{2n}u^{-\langle w, w - 2\hat\varepsilon_1\rangle + 1} \prod_{v \in A[2]} \exp\big((\textstyle\int_w\gamma_v)y_v\big)q^w\\
 &= \lambda_0 \frac{u^2}{\Delta(u^2)}\sum_{w \in K} u^{-\langle w, w - 2\hat\varepsilon_0\rangle} \prod_{v \in A[2]}\exp \big((\textstyle\int_w(\gamma_v)y_v\big)q^w \\
 &\quad + \lambda_1 \frac{u^2}{\Delta(u^2)}\sum_{w \in K} u^{-\langle w, w - 2\hat\varepsilon_1\rangle+1} \prod_{v \in A[2]}\exp \big((\textstyle\int_w(\gamma_v)y_v\big)q^w
\end{align*}
as claimed.
\end{proof}

\subsection{Application of the crepant resolution conjecture}\label{subsec_crc}

We now use apply the crepant resolution conjecture of \cite{crc} to compute the potential function of the orbifold $[A/\pm1]$. This conjecture is given as follows. 

\begin{conjecture}[{\cite[Conjecture 1.2]{crc}}]
Given an orbifold $\mathscr{X}$ satisfying the hard Lefschetz condition and admitting a crepant resolution $Y$, there exists a graded linear isomorphism
\[
L : H^*(Y) \to H^*_{orb}(\mathscr{X})
\]
and roots of unity $c_1, \ldots, c_r$ such that the following conditions hold.
\begin{enumerate}
\item The inverse of $L$ extends the map $\pi^* : H^*(\mathscr{X}) \to H^*(Y)$.
\item Regarding the potential function $F^Y$ as a power series in $y_0, \ldots, y_a$, and in $q_1, \ldots, q_s$, the coefficients admit analytic continuations from $(q_{s+1}, \ldots, q_r) = (0, \ldots, 0)$ to $(q_{s+1}, \ldots, q_r) = (c_{s+1}, \ldots c_r)$.
\item The potential functions $F^\mathscr{X}$ and $F^Y$ are equal after the substitution
\[
y_i = \sum_j L_i^j x_j \qquad \qquad q_i = \begin{cases} c_i & \text{for } i > s \\ c_it_i & \text{for } i \leq s \end{cases}
\]
\end{enumerate}
\end{conjecture}



It should be noted that the crepant resolution conjecture, as stated, only applies to ordinary Gromov-Witten invariants and not to reduced Gromov-Witten invariants as we use in our case. The short explanation is simply that it appears to work. The longer explanation is that our situation (dealing with fibre-wise Gromov-Witten invariants of a non-K\"ahler 3-fold) is similar enough to the CY3 case that it {\em should} work. Moreover, the local picture around the singular points is the same as that of the resolution $T^*\PP^1 \to [\C^2/\pm1]$, a case where the equivariant version of the crepant resolution conjecture has been proven \cite{crc}. The only issue is whether or not we can extant this globally to the orbifold $[A/\pm1]$, which appears to be the case.

As is usually the case, the change-of-variables is forced upon us by knowing a few of the invariants and choosing data to match those. In our case this comes from two sources. As stated above, we have a good understanding of the local picture around a singular point, so we can use that knowledge. Moreover, we already know the number of genus 2 curves in an Abelian surface due to G\"ottsche \cite{gottsche}. These two facts allow us to derive the full change-of-variables for the restricted potential function, which is given as follows.

The function $L : H^*\big(\Km(A)\big) \to H^*_{orb}([A/\pm1])$ used in the crepant resolution conjectures is the following. Let $y_v$ denote the formal cohomological variable corresponding to the dual of an exceptional divisor $E_v$, and let $z_v$ denote the variable corresponding to the class in the twisted sector of $[A/\pm1]$. Similarly, let $y_S$ and $y_F$ denote the formal variables corresponding to the classes $S$ and $F$ (with $z_S$ and $z_F$ downstairs), respectively. Then the map $L$ on the formal variables is given by
\[
L(y_v) = iz_v \qquad L(y_S) = z_S \qquad L(y_F) = z_F
\]
while the roots of unity are given by applying, for $w = \sum_{v \in A[2]}\tfrac{a_v}{2}E_v$
\[
q^w \mapsto (-1)^w := (-1)^{\sum_{v \in A[2]} \tfrac{a_v}{2}}
\]
and the substitutions
\[
q^{S/2} \mapsto t^{S/2} \qquad q^{F/2} \mapsto - t^{F/2}
\]
As in the statement of Theorem \ref{thm_restricted_potential_km}, we will continue with the equivalent substitution of
\[
u =  \big(\exp (z_F) t^F\big)^{1/2}.
\]
which yields that $u \mapsto -u$ under this substitution. All of this together yields the following.

\begin{proposition}
Assume the crepant resolution conjecture holds for the resolutions $\Km(A_n) \to [A_n/\pm1]$. Then the restricted genus 0 orbifold Gromov-Witten potential function for $[A/\pm1]$, when summed over all polarizations, is given by
\begin{multline*}
F^{[A/\pm1]} = \big(\exp(z_S)t^S\big)^{1/2}\frac{u^2}{\Delta(u^2)}\quad\cdot\\ \Big(\prod_{v\in\varepsilon_0} e^{-iz_v/2}\sum_{w \in K}u^{-\langle w, w - 2\hat\varepsilon_0\rangle}\prod_{v\in A[2]}\exp\big(i(\textstyle\int_w\gamma_v)z_v\big)(-1)^w \\
+ \prod_{v\in\varepsilon_1} e^{-iz_v/2}\sum_{w \in K}u^{-\langle w, w - 2\hat\varepsilon_1\rangle+1}\prod_{v\in A[2]}\exp\big(i(\textstyle\int_w\gamma_v)z_v\big)(-1)^w\Big).
\end{multline*}
where $(-1)^w := (-1)^{\sum_{v \in A[2]} \tfrac{a_v}{2}}$ for $w = \sum_{v \in A[2]}\tfrac{a_v}{2}E_v$.
\end{proposition}

\begin{proof}
If we combine all of the transformations given above, we find that the remaining composite terms in the potential transform as
\[
\lambda_0 \mapsto \big(\exp(z_S)t^S\big)^{1/2}\prod_{v \in \varepsilon_0} e^{iz_v/2} \qquad\qquad \lambda_1 \mapsto -\big(\exp(z_S)t^S\big)^{1/2}\prod_{v \in \varepsilon_1} e^{iz_v/2}
\]
This yields the result as claimed.
\end{proof}

\begin{remark}
We should remark that for simplicity, we will omit the term $\big(\exp(z_S)t^S\big)^{1/2}$ occurring at the beginning of the expression, which contributes no further enumerative information.
\end{remark}

To simplify this, we recall that from corollary \ref{cor_kum_lat} we can write any element $\bar{w} \in K$ as $\bar{w} = w + \hat\eta$ for $w \in \Lambda$ and $\eta \in \Pi_3$. This allows us to replace the summation over $K$ by a double summation over $\Lambda$ and over $\Pi_3$. More precisely, let $\eta \in \Pi_3$, and (noting that $(-1)^{\hat\eta} = 1$ by the definition above) define
\begin{align*}
F^{[A/\pm1]}_\eta &= \frac{u^2}{\Delta(u^2)}\Big(\prod_{v\in\varepsilon_0} e^{-iz_v/2}\sum_{w \in \Lambda}u^{-\langle w+\hat\eta, w+\hat\eta - 2\hat\varepsilon_0\rangle}\quad\cdot \\
& \qquad\qquad\qquad\qquad \prod_{v\in A[2]}\exp\big(i(\textstyle\int_{w+\hat\eta}\gamma_v)z_v\big)(-1)^w \\
&\quad + \prod_{v\in\varepsilon_1} e^{-iz_v/2}\sum_{w \in \Lambda}u^{-\langle w+\hat\eta, w+\hat\eta - 2\hat\varepsilon_1\rangle+1} \quad\cdot\\
&\qquad \qquad\qquad\qquad\prod_{v\in A[2]}\exp\big(i(\textstyle\int_{w+\hat\eta}\gamma_v)z_v\big)(-1)^w\Big).
\end{align*}
Then we have that $F^{[A/\pm1]} = \sum_{\eta \in \Pi_3} F^{[A/\pm1]}_\eta$. Since the intersection form restricted to $\Lambda$ is diagonal, the functions $F^{[A/\pm1]}_\eta$ can now be computed. For simplicity, the superscript $[A/\pm1]$ will now be omitted.

\begin{lemma}\label{lem_f_eta}
The function $F_\eta$ can be written as
\begin{multline*}
F_\eta = \frac{u^2}{\Delta(u^2)}\Big(u^{\tfrac{1}{2}|\eta + \varepsilon_0|-2}\prod_{v \in \eta + \varepsilon_0}h(z_v,u)\prod_{v\notin\eta + \varepsilon_0}g(z_v, u) \\
+ u^{\tfrac{1}{2}|\eta + \varepsilon_1|-2}\prod_{v \in \eta + \varepsilon_1}h(z_v,u)\prod_{v\notin\eta + \varepsilon_1}g(z_v, u)\Big)
\end{multline*}
where
\begin{gather*}
h(z,u) = 2\sum_{k=0}^\infty (-1)^k\sin\big((2k+1)\frac{z}{2}\big)u^{2k^2+2k}\\
g(z,u) = 1+ 2\sum_{k=1}^\infty (-1)^k\cos(kz)u^{2k^2}.
\end{gather*}
\end{lemma}

\begin{proof}
Throughout this proof we will only work with the terms involving $\varepsilon_0$; all of the proofs for $\varepsilon_1$ are nearly identical. We begin by noting that
\[
\langle w + \hat\eta, w + \hat\eta - 2\hat\varepsilon_0\rangle = \langle w, w + 2\hat\eta - 2\hat\varepsilon_0\rangle + \langle \hat\eta, \hat\eta \rangle - 2\langle \hat\eta, \hat\varepsilon_0\rangle.
\]
Since $|\varepsilon_0| = 4$, and for any two subsets $\eta_1, \eta_2 \in \Pi_0$, we have $\langle \hat\eta_1, \hat\eta_2 \rangle = -\tfrac{1}{2}|\eta_1 \cap \eta_2|$ (see remark \ref{rmk_int_form}), it follows that
\begin{align*}
\langle \hat\eta, \hat\eta \rangle - 2\langle \hat\eta, \hat\varepsilon_0\rangle 
&= \langle \hat\eta, \hat\eta \rangle - 2\langle \hat\eta, \hat\varepsilon_0\rangle + \langle \hat\varepsilon_0, \hat\varepsilon_0\rangle + 2 \\
&= -\frac{1}{2}\big(|\eta| - 2 |\eta \cap \varepsilon_0| + |\varepsilon_0|\big) + 2 \\
&= -\frac{1}{2}|\eta + \varepsilon_0| + 2
\end{align*}
where we recall that the summation is done in $\Pi_0$ (and so is the symmetric difference of the two sets).

We work now to simplify the expression
\[
\prod_{v\in\varepsilon_0} e^{-iz_v/2}\sum_{w \in \Lambda}u^{-\langle w, w+2\hat\eta - 2\hat\varepsilon_0\rangle}\prod_{v\in A[2]}\exp\big(i(\textstyle\int_{w+\hat\eta}\gamma_v)z_v\big)(-1)^w.
\]
To begin with, we note that $\exp\big(i(\int_{\hat\eta}\gamma_v)z_v\big) = \exp(iz_v/2)$ if $v \in \eta$ (and is 1, otherwise), and so this term is
\[
\prod_{v\in\varepsilon_0\setminus\eta} e^{-iz_v/2}\prod_{v\in\eta\setminus\varepsilon_0}e^{iz_v/2}\sum_{w \in \Lambda}u^{-\langle w, w+2\hat\eta - 2\hat\varepsilon_0\rangle}\prod_{v \in A[2]}\exp\big(i(\textstyle\int_w\gamma_v)z_v\big)(-1)^w
\]
Because of the fact that the intersection form restricted to $\Lambda$ is diagonal, we can write the latter sum as a product
\begin{align*}
\sum_{w \in \Lambda}u^{-\langle w, w+2\hat\eta - 2\hat\varepsilon_0\rangle}\prod_{v\in A[2]}\exp\big(i(\textstyle\int_w\gamma_v)z_v\big)(-1)^w
&= \prod_{v \in \eta\setminus\varepsilon_0}\sum_{k=-\infty}^\infty (-1)^ku^{2k^2+2k}e^{ikz_v}
\\
&\quad \cdot \prod_{v \in \varepsilon_0\setminus\eta}\sum_{k=-\infty}^\infty (-1)^ku^{2k^2-2k}e^{ikz_v} \\
&\quad \cdot \prod_{v \notin \eta + \varepsilon_0} \sum_{k=-\infty}^\infty (-1)^ku^{2k^2}e^{ikz_v}.
\end{align*}
By rearranging the summation index on the middle product it can be written as
\[
\prod_{v \in \varepsilon_0\setminus\eta}\sum_{k=-\infty}^\infty (-1)^ku^{2k^2-2k}e^{ikz_v} = \prod_{\varepsilon_0\setminus\eta}\big(-e^{iz_v}\big)\sum_{k=-\infty}^\infty(-1)^ku^{2k^2+2k}e^{ikz_v}
\]
and so the whole term becomes
\begin{multline*}
\prod_{v\in\varepsilon_0\setminus\eta} e^{-iz_v/2}\prod_{v\in\eta\setminus\varepsilon_0}e^{iz_v/2}\sum_{w \in \Lambda}u^{-\langle w, w+2\hat\eta - 2\hat\varepsilon_0\rangle}\prod_{v\in A[2]}\exp\big(i(\textstyle\int_w\gamma_v)z_v\big)(-1)^w \\
= \prod_{v\in\eta+\varepsilon_0} e^{iz_v/2}\sum_{k=-\infty}^\infty (-1)^ku^{2k^2+2k}e^{ikz_v}\prod_{v\notin\eta+\varepsilon_0}\sum_{k=-\infty}^\infty (-1)^ku^{2k^2}e^{ikz_v}
\end{multline*}

Since we have that
\[
\sum_{k=-\infty}^\infty (-1)^ku^{2k^2}e^{ikz} = 1+2\sum_{k=1}^\infty (-1)^k\cos(kz)u^{2k^2}
\] 
and also that
\[
e^{iz/2}\sum_{k=-\infty}^\infty (-1)^ku^{2k^2+2k}e^{ikz} = 2\sum_{k=0}^\infty \sin\big((2k+1)z/2\big)u^{2k^2+2k}
\]
the conclusion follows.
\end{proof}

\subsection{Proof of Theorem \ref{thm_main}}\label{subsec_pf}

To prove Theorem \ref{thm_main} it remains to study the functions $h$ and $g$ given above. As they are written in terms of trigonometric functions, we begin by writing them in terms of Chebyshev polynomials.

Recall that Chebyshev polynomials $T_n(x)$ are given by the relationship
\[
T_n(\cos\theta) = \cos(n\theta).
\]
They can equivalently be defined recursively via
\begin{gather}
T_0(x) = 1 \notag\\
T_1(x) = x \notag\\
T_n(x) = 2xT_{n-1}(x) - T_{n-2}(x). \label{cheb_recurrence}
\end{gather}
They also satisfy the relationships
\[
T_{2n+1}(\sin\theta) = (-1)^n\sin\big((2n+1)\theta\big)
\]
and
\[
T_n(1-2x^2) = (-1)^nT_{2n}(x).
\]
With these last two in mind, we have the following fact.

\begin{proposition}\label{prop_h_g}
The functions $h(x,u)$ and $g(x,u)$ can be written as
\begin{gather*}
h(z,u) = 2\sum_{n=0}^\infty T_{2n+1}\big(\sin(z/2)\big)u^{2n^2+2n} \\
g(z,u) = 1 + 2\sum_{n=1}^\infty T_{2n}\big(\sin(z/2)\big)u^{2n^2}
\end{gather*}
\end{proposition}

Proposition \ref{prop_h_g} and Lemma \ref{lem_f_eta} give us an explicit description for the (relevant portion of the) potential function of the orbifold $[A/\pm1]$. We now perform the substitution $x_v = 2\sin(z_v/2)$ (See Theorem \ref{thm_substitution}), and use the following result from \cite{andrews_rose}.

\begin{theorem}\label{thm_andrews_rose}
Let $H(x,q)$ be defined as $2\sum_{n=0}^\infty T_{2n+1}(x/2)q^{n^2+n}$ and let $G(x,q)$ be defined as $1 + 2\sum_{n=1}^\infty T_{2n}(x/2)q^{n^2}$. Then
\begin{gather*}
H(x,q) = (q^2;q^2)^3_\infty \sum_{k=0}^\infty A_k(q^2) x^{2k+1} \\
G(x,q) = \frac{(q;q)_\infty}{(-q;q)_\infty}\sum_{k=0}^\infty C_k(q)x^{2k}
\end{gather*}
where $(a;q)_\infty = \prod_{k=0}^\infty (1 - aq^k)$ is a $q$-Pockhammer symbol, and $A_k, C_k$ are MacMahon's generalized sum-of-divisors functions (See \cite{macmahon}).
\end{theorem}

\begin{theorem}\label{thm_qmod}
The functions $A_k$ and $C_k$ satisfy the recursive definitions
\begin{gather*}
A_1(q) = \sum_{k=1}^\infty \sigma_1(k)q^k \\
C_1(q) = A_1(q) - A_1(q^2) \\
A_k(q) = \frac{1}{(2k+1)2k}\Big(\big(6A_1(q) + k(k-1)\big)A_{k-1}(q) - 2q\frac{d}{dq}A_{k-1}(q)\Big) \\
C_k(q) = \frac{1}{2k(2k-1)}\Big(\big(2C_1(q) + (k-1)^2\big)C_{k-1}(q) - q\frac{d}{dq}C_{k-1}(q)\Big)
\end{gather*}
and as such are in the ring of quasi-modular forms.
\end{theorem}

We are now in a position to prove theorem \ref{thm_main}.

\begin{proof}[Proof of theorem \ref{thm_main}]
From Lemma \ref{lem_f_eta}, Proposition \ref{prop_h_g}, and Theorem \ref{thm_andrews_rose}, we can write
\begin{align*}
F_\eta &= \frac{u^2}{\Delta(u^2)}\Big(u^{\tfrac{1}{2}|\eta +\varepsilon_0|-2}\prod_{v\in\eta+\varepsilon_0}(u^4;u^4)_\infty^3\sum_{k=0}^\infty A_k(u^4)x_v^{2k+1} \quad\cdot\\
&\qquad\qquad\qquad\qquad\qquad\qquad\prod_{v\notin \eta + \varepsilon_0} \frac{(u^2;u^2)_\infty}{(-u^2;u^2)_\infty}\sum_{k=0}^\infty C_k(u^2)x_v^{2k} \\
 &\quad + u^{\tfrac{1}{2}|\eta + \varepsilon_1|-2}\prod_{v\in\eta+\varepsilon_1}(u^4;u^4)_\infty^3\sum_{k=0}^\infty A_k(u^4)x_v^{2k+1}\quad\cdot\\
 &\qquad\qquad\qquad\qquad\qquad\qquad\prod_{v\notin \eta + \varepsilon_1} \frac{(u^2;u^2)_\infty}{(-u^2;u^2)_\infty}\sum_{k=0}^\infty C_k(u^2)x_v^{2k}\Big)
\end{align*}

Collecting all the Pockhammer symbols (and the extra powers of $u$) we find (for the $\varepsilon_0$ term) 
\[
u^{\tfrac{1}{2}|\eta + \varepsilon_0| -2}\frac{u^2}{\Delta(u^2)}\big((u^4;u^4)_\infty^3\big)^{|\eta + \varepsilon_0|}\Bigg(\frac{(u^2;u^2)_\infty}{(-u^2;u^2)_\infty}\Bigg)^{16 - |\eta + \varepsilon_0|}.
\]
Using the fact that $\frac{u^2}{\Delta(u^2)} = (u^2;u^2)_\infty^{-24}$ and that $(q;q)_\infty(-q;q)_\infty = (q^2;q^2)_\infty$ this simplifies to
\[
\Big(u\big((u^2;u^2)_\infty(-u^2;u^2)_\infty^2\big)^4\Big)^{\tfrac{1}{2}|\eta + \varepsilon_0| - 2}
\]
Using a theorem of Legendre \cite{leg-1828} which states that $\big((q;q)_\infty(-q;q)^2_\infty\big)^4 = \sum_{k=0}^\infty \sigma_1(2k+1)q^k$, it follows that this term is in fact $E(u)^{\tfrac{1}{2}|\eta + \varepsilon_0| - 2}$ where
\[
E(u) = \sum_{k=0}^\infty \sigma_1(2k+1)u^{2k+1} = \frac{1}{16}\vartheta_2(u)^4
\]
and where $\vartheta_2$ is the Jacobi theta function $\vartheta_2(q) = \sum_{k=0}^\infty q^{(k+1/2)^2}$. An identical computation yields that the terms in front of the $\varepsilon_1$ term simplify to $E(u)^{\tfrac{1}{2}|\eta + \varepsilon_1|-2}$.

We can now rewrite $F_\eta$ as
\begin{multline*}
F_\eta = E(u)^{\tfrac{1}{2}|\eta + \varepsilon_0|-2}\prod_{v\in\eta+\varepsilon_0}\sum_{k=0}^\infty A_k(u^4)x_v^{2k+1}\prod_{v\notin \eta+\varepsilon_0}\sum_{k=0}^\infty C_k(u^2)x_v^{2k} \\
+ E(u)^{\tfrac{1}{2}|\eta + \varepsilon_1|-2}\prod_{v\in\eta+\varepsilon_1}\sum_{k=0}^\infty A_k(u^4)x_v^{2k+1}\prod_{v\notin \eta+\varepsilon_1}\sum_{k=0}^\infty C_k(u^2)x_v^{2k}.
\end{multline*}

This is the generating function for counting hyperelliptic curves $f: \Sigma \to A_n$ with the property that $f(w) \in A[2]$ for all Weirstrass points $w$ where we include the data of how many Weirstrass points lie on a given 2-torsion point. Specifically, the coefficient of the monomial $\prod_{v\in A[2]} x_v^{\bk(v)}$ gives the number of such curves with $\bk(v)$ Weirstrass points having as image the point $v \in A[2]$. From this description, the conclusion follows.
\end{proof}

\begin{corollary}
The coefficient of a monomial $\prod_{v \in A[2]}x_v^{\bk(v)}$ in the generating function $F$ lies in the ring of quasimodular forms.
\end{corollary}

\begin{proof}
This follows directly from Theorem \ref{thm_main} and from Theorem \ref{thm_qmod}.
\end{proof}

\subsection{A simple consequence}

One consequence of this formula is a prediction for the lower bound on the arithmetic genus of a hyperelliptic curve in an Abelian surface $A$, which we prove.

Note that the lowest degree terms of $A_k(q)$ and $C_k(q)$, respectively, are $\binom{k+1}{2}$ and $k^2$. Thus the lowest degree term of
\[
E(u)^{\tfrac{1}{2}|S| - 2}\prod_{v\in S}A_{\frac{\bk(v)-1}{2}}(u^4)\prod_{v \notin S} C_{\frac{\bk(v)}{2}}(u^2)
\]
(see equation \eqref{eq_main}) is given by
\[
-2 + \sum_{v \in S} \tfrac{1}{2}  + \sum_{v \in S} 4\binom{\tfrac{1}{2}(\bk(v) - 1)+1}{2} + \sum_{v \notin S} 2 \Big(\frac{\bk(v)}{2}\Big)^2 = -2 + \sum_{v \in A[2]}\tfrac{1}{2} \bk(v)^2.
\]
It follows that this formula predicts that the minimal arithmetic genus of a hyperelliptic curves in $A$ (with discrete data $\bk$) is
\[
-1 + \sum_{v \in A[2]} \tfrac{1}{2}\bk(v)^2
\]
by the definition of $F_{g, \bk}(u)$. We now show {\em without assuming the crepant resolution conjecture} that this is true.

\begin{theorem}
Fix $\bk$, and let $g$ be such that $2g + 2 = |\bk|$, and suppose that $A$ has Picard number 1. Then the minimal arithmetic genus of a hyperelliptic curve in $A$ with discrete data $\bk$ is $\sum_{v \in A[2]} \tfrac{1}{2} \bk(v)^2 - 1$.
\end{theorem}

\begin{proof}
Given such discrete data, the geometric genus is given by $g = \tfrac{1}{2}|\bk| - 1$. For each $v \in A[2]$ with $\bk(v) > 1$, we note that any curve which produces such data must be nodal, since more than one Weirstrass point will have the same image. More specifically, we introduce at least $\binom{\bk(v)}{2}$ nodes for each $\bk(v) > 1$. As such, the total arithmetic genus is given by
\begin{align*}
\tfrac{1}{2}|\bk| - 1 + \sum_{v \in A[2]}\binom{\bk(v)}{2} 
 &= \tfrac{1}{2}\sum_{v \in A[2]} \bk(v) - 1 + \sum_{v \in A[2]} \tfrac{1}{2}\big(\bk(v)^2 - \bk(v)\big) \\
 &= -1 + \sum_{v\in A[2]} \tfrac{1}{2} \bk(v)^2
\end{align*}
as claimed.
\end{proof}

\begin{corollary}
There are no smooth hyperelliptic curves in $A$ of genus greater than 5.
\end{corollary}

\begin{proof}
We first note that if any $\bk(v) > 1$,  then there must be at least one node in the image curve, as we have two Weirstrass points with the same image, $v$.

Next, note that due to the requirement that $P$, the set of those $v$ with $\bk(v)$ odd must be congruent to $\varepsilon_i$ mod $K$ (or more accurately, mod $\Pi_3$) yields that no more than twelve of the sixteen 2-torsion points can have $\bk(v) = 1$ (this maximal case is when $P$ is the complement of $\varepsilon_0$). In this case, we have that $2g+2 = 12$, or $g = 5$.
\end{proof}


\section{Proofs for Low Genera}\label{sec_low_genera}

\subsection{Introduction}

We will prove the formula \eqref{eq_main} for the case of genus 1 and genus 2 curves. Recall that a hyperelliptic curve of genus $g$ yields a $\PP^1$ with $2g + 2$ stacky points, and so to enumerate genus $g$ curves we need only consider those monomials of total degree $2g + 2$. For genus 1, this is degree 4 (we will explain below what a genus one curve in $A$, which should generically have no elliptic curves, means), and for genus 2, this is of degree 6.

We assume as in Section \ref{sec_specialization} that the Abelian surface is a product $A \cong S \times F$ of non-isogenous elliptic curves. Recall that there is then an elliptic fibration $A/\pm1 \to S/\pm1$ with general fibre $F$, and with four special fibres $F/\pm1$ over the points $v \in S[2]$. We will further (by abuse of notation) denote by $S$ and by $F$ the classes in $A/\pm1$ of a section $S/\pm1$ and fibre $F/\pm1$, respectively.


In the case that the Picard number of $A$ is 1 and that the curve class is primitive, the source curves are combs with collapsing teeth. This allows us to perform the substitution given in in Theorem \ref{thm_substitution}. When the Picard number of $A$ is greater than 1---as is the case when $A \cong S \times F$---or when the class is not primitive, it is possible that there are collapsing components which join two components mapped into $A$ with non-zero degree (see figure \ref{fig_collapse_between}).

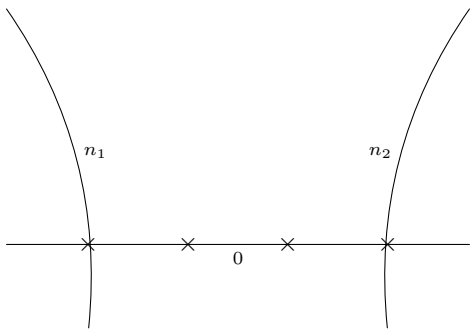
\begin{figure}[ht]
\[
\xymatrix{
 & & & & & \\
 & & & & & \\
 & & & & & \\
 \ar@{-}[rrrrr]_{0} & \times & \times & \times & \times & \\
 & \ar@{-}@/_1.06pc/[luuuu]_{n_1}& & & \ar@{-}@/^1.06pc/[ruuuu]^{n_2} & \\
}
\]
\caption{A collapsing component joining two non-collapsing ones}\label{fig_collapse_between}
\end{figure}

These do not contribute to the Gromov-Witten invariant (see Proposition \ref{prop_double_psi}). From these two considerations, we can consider only those maps which do not collapse any components.

Let $f : \Sigma \to [A/\pm1]$ be a rational curve. As we saw, the source is a tree of rational curves $\Sigma = \Sigma_0 \cup \Sigma_1 \cup \cdots \cup \Sigma_k$, and such that $f_*[\Sigma_0] = S$ and $f_*[\Sigma_i] = n_iF$ for $i \geq 1$ with $\sum n_i = n$ (and $n_i > 0$, as discussed above). Moreover, the leaves of the tree of curves must have at least 3 stacky points on them (plus a node which connects them to the tree).

For convenience, we now need to label the 2-torsion points of $A$. We label them as
\begin{equation}
\xy
\xymatrix@R0.5cm@C0.5cm{
\text{2-torson from $F$}\ar[r] & E_3 \save[0,0].[3,0]!C *\frm<8pt>{-} \restore \save[0,0].[0,3]!C *\frm<8pt>{.} \restore& E_7 & E_{11} & E_{15} \\
 & E_2 \save[0,0].[0,3]!C *\frm<8pt>{.} \restore& E_6 & E_{10} & E_{14} & \text{Translates of $S$}\ar[l]\ar[ld]\ar[lu]\\
 & E_1 \save[0,0].[0,3]!C *\frm<8pt>{.} \restore& E_5 & E_9 & E_{13} \\
 & E_0 \save[0,0].[0,3]!C *\frm<8pt>{-} \restore & E_4 & E_8 & E_{12} & \text{2-torsion from $S$}\ar[l]
}
\endxy
\end{equation}
with monomials labeled similarly. With this labelling, we have the explicit descriptions $\varepsilon_0 = \tfrac{1}{2}(E_0 + E_4 + E_8 + E_{12})$ and $\varepsilon_1 = \tfrac{1}{2}(E_1 + E_2 + E_3) + \tfrac{1}{2}(E_4 + E_8 + E_{12})$.

\subsection{Genus 1}

The genus 1 case is somewhat of an aberation; there are of course no elliptic curves in a generic Abelian surface. It turns out that the genus 1 case corresponds to a `polarization' of type $(1, 0)$. By Poincar\'e's Reducibility Theorem \cite[Chapter 5, Theorem 3.5]{cav}, such an $A$ is isogenous to a product of elliptic curves.


Since the source curve is genus 1, on the orbifold side we have exactly 4 marked stacky points. From the description above, if the source curve were to consist of more than one component, it would have at least 6 marked points, and so the source curve must be irreducible with $f$ mapping it isomorphically onto a section $S/\pm1$. There is only one such map (as $S, F$ were chosen to be generic).

The only term with monomials of degree 4 from all of the functions $F_\eta$ are $x_0x_4x_8x_{12}$, $x_1x_5x_9x_{13}$, $x_2x_6x_{10}x_{14}$, and $x_3x_7x_{11}x_{15}$. If we look at the formula \eqref{eq_main}, the prediction for each of these is also 1, and so the conjecture is verified.

It should be noted that in terms of computing the actual number of such curves, these latter three monomials yield curves that are simply translations of the first, and so the number of genus 1 curves in $A$ in the class $S$ is 1, as we would expect.

\subsection{Genus 2}

The component which maps isomorphically onto a section curve $S/\pm1$ must have three of the six marked points, and so our source curve must consist of exactly two components, $\Sigma = \Sigma_0 \cup \Sigma_1$, and where $f_*[\Sigma_1] = nF$. Thus the component $\Sigma_1$ is an $n$-fold cover of $F/\pm1$ (with certain data about the images of the stacky points in $\Sigma_1$), and so we are reduced to computing the number of such covers.

Analogous to the genus 1 case, since we only care {\em up to translation in $A$}, we need only focus on certain monomials. These are
\[
x_1x_2x_3x_4x_8x_{12} \quad x_0^3x_4x_8x_{12} \quad x_0x_4x_8x_{12}x_1^2 \quad x_0x_4x_8x_{12}x_2^2 \quad x_0x_4x_8x_{12}x_3^2
\]
and the corresponding predictions are
\[
E(u) \qquad A_1(u^4) \qquad C_1(u^2) \qquad C_1(u^2) \qquad C_1(u^2).
\]
The first comes simply from the term $\varepsilon_1$ and the others all come from $\varepsilon_0$.

In each case, we are counting maps $\Sigma_1 \to [F/\pm1]$ of a curve with four $\Z/2$-points to the orbifold $[F/\pm1]$. In this case, we can lift to the cover
\[
\xymatrix{
E \ar[r]\ar[d] & F \ar[d]\\
\Sigma_1 \ar[r] & F/\pm1
}
\]
and so we can compute this by counting covers of $F$ by elliptic curves $E$ satisfying certain conditions based on the image of the 2-torsion of $E$. These are the following.
\begin{enumerate}
\item The first monomial corresponds to those maps $E \to F$ such that the 2-torsion of $E$ surjects onto the 2-torsion of $F$.
\item The second monomial corresponds to those maps such that all 2-torsion of $E$ maps to $0 \in F$.
\item The last three monomials all correspond to the three possible cases where $E[2]$ surjects onto a subgroup $\langle e_i\rangle$ for $e_i$ a non-zero 2-torsion point in $F$.
\end{enumerate}

We first need a classically known fact.

\begin{proposition}
Let $E$ be an elliptic curve. Then the number of degree $n$ isogenies $F \to E$ is given by $\sigma_1(n) = \sum_{d \mid n} d$.
\end{proposition}

\begin{proof}
As maps of Abelian varieties are determined by their lattices, counting degree $n$ isogenies $F \to E$ is the same as counting index $n$ sublattices of a fixed rank 2 lattice. Up to change of basis of the respective lattices, this is the same as counting matrices of the form
\[
\begin{pmatrix}
a & b \\ 0 & d
\end{pmatrix}
\]
with $ad = n$ and $ 0 \leq b < d$. This is clearly equal to $\sigma_1(n)$ as claimed.
\end{proof}

All that is now required are the following lemmata. We will provide a proof of Lemma \ref{lem_isogeny}, the rest of them having a similar flavour.

\begin{lemma}
Suppose that $g : E \to F$ is a map of elliptic curves. Then the degree of $g$ is odd if and only if the 2-torsion surjects.
\end{lemma}

In this case the generating function is simply the odd-degree part of $A_1(u)$, i.e. $E(u)$.

\begin{lemma}
Suppose that $g : E \to F$ is a map of elliptic curves such that $g(E[2]) = 0$. Then $g$ factors through the degree 4 map
\[
\xymatrix{
E \ar[rr]^g\ar@{.>}[rd] && F \\
& F/F[2]\ar[ru]
}
\]
\end{lemma}

In this case, we are actually counting all covers of the curve $F/F[2]$---but this is of course the same as counting covers of degree $n/4$ of an arbitrary curve. That is, we obtain the function $A_1(u^4)$ as desired.

\begin{lemma}\label{lem_isogeny}
Suppose that $g :E \to F$ is a map of elliptic curves such that $g(E[2])$ surjects onto the subgroup $\langle v\rangle$ for some non-zero $v \in F[2]$. Then $g$ factors through the degree 2 map
\[
\xymatrix{
E \ar[rr]^g\ar@{.>}[rd] && F \\
& F/\langle v\rangle\ar[ru]
}
\]
\end{lemma}

\begin{proof}
As before, we consider maps of elliptic curves via the maps on their underlying lattices. That is, we consider the map $\tilde g : \Lambda_E \hookrightarrow \Lambda_F$. Choose a basis $\lambda_1, \lambda_2$ of $\Lambda_F$, and $\mu_1, \mu_2$ of $\Lambda_E$ such that $\frac{1}{2}\lambda_1$ represents the 2-torsion point $v$. We will show that its image lies in the sublattice $\langle \lambda_1, 2\lambda_2\rangle$, which proves the lemma.

The condition above yields (up to labelling) that
\[
g\big(\tfrac{1}{2}\mu_1\big) \equiv \tfrac{1}{2} \lambda_1 \pmod {\Lambda_F} \qquad \qquad \text{and} \qquad\qquad g\big(\tfrac{1}{2}\mu_2\big) \in \Lambda_F
\]
or equivalently that
\[
g(\mu_1) - \lambda_1 \in 2\Lambda_F \qquad\qquad \text{and}\qquad\qquad g(\mu_2) \in 2\Lambda_F.
\]
It follows that $g(\Lambda_E) \subset \langle \lambda_1, 2\Lambda_F\rangle = \langle \lambda_1, 2\lambda_2\rangle$ as claimed.
\end{proof}

In this final case, we are counting those covers of degree $n/2$ of $F/\langle e_i \rangle$ (i.e. $A_1(u^2)$) less those whose 2-torsion is all mapped to zero (i.e. $A_1(u^4)$, from before). That is, the generating function is $A_1(u^2) - A_1(u^4) = C_1(u^2)$.

These three lemmas together prove the conjecture for the genus 2 case.

\subsection{Proof of G\"ottsche's genus 2 formula}\label{sec_gottsche}

If we consider all the previous cases, it follows that the number of genus 2 curves in $A$ up to translation in $A$ is given by the sum of all of the given terms above. That is, if we let $F_2(u)$ denote the number of genus 2 curves in $A$ up to translation, then we have
\begin{align}
F_2(u) &= E(u) + A_1(u^4) + 3\big(A_1(u^2) - A_1(u^4)\big) \notag \\
 &= E(u) + 3A_1(u^2) - 2A_1(u^4) \label{eq_gottsche}.
\end{align}
It is clear that the odd powers of $u$ in the right-hand side are the same as those in $A_1(u)$. The fact that the even powers match those of $A_1(u)$ will follow from the following lemmata. We once more prove only the second, the first being similar.

\begin{lemma}
Let $n > 0$ be congruent to $2 \pmod 4$. Then
\[
\sigma_1(n) = 3\sigma_1(n/2).
\]
\end{lemma}

\begin{lemma}
Let $n > 0$ be congruent to $0 \pmod 4$. Then
\[
\sigma_1(n) = 3\sigma_1(n/2) - 2\sigma_1(n/4).
\]
\end{lemma}

\begin{proof}
Write $n = 2^km$ with $2 \nmid m$ and with $k \geq 2$. Then as $\sigma_1$ is a multiplicative function, we have that
\begin{align*}
\sigma_1(n) &= \sigma_1(2^k)\sigma_1(m) \\
 &= (2^{k+1}-1)\sigma_1(m) \\
 &= \big(3(2^k - 1) - 2(2^{k-1} - 1)\big)\sigma_1(m) \\
 &= 3\sigma_1(n/2) - 2\sigma_1(n/4)
\end{align*}
as claimed.
\end{proof}

In \cite{gottsche}, the following is proven.

\begin{theorem}[G\"ottsche, Theorem 3.2]\label{thm_gottsche}
Let $(A,L)$ be a polarized Abelian surface with polarization of type $(1, n)$. Then the generating function for the number of genus 2 curves in the linear system $|L|$, summed over all polarization types, is given by
\[
\tilde{F}_2(u) = \sum_{n=1}^\infty n^2\sigma_1(n) u^n = D^2A_1(u)
\]
\end{theorem}

To see that these are equivalent, we note that the difference between the two counts---that is, curves in a fixed linear system vs. curves up to translation---comes from translating by elements in the kernel of the isogeny $A \to \widehat{A} = Pic^0(A)$ given by $a \mapsto L \otimes t_a^*L^{-1}$. If the polarization is of type $(1, n)$, then the map on lattices $H_1(A) \to H_1(\hat A)$ can be represented by the matrix \eqref{eq_polarized_basis} (see Section \ref{sec_ab_kum}).  The map is thus of degree $n^2$, and so the kernel consists of exactly $n^2$ elements. This yields the claim.

\appendix

\section{Structure of the Moduli Space}\label{appendix_mod}

In this Appendix we gather a few facts about the structure of the moduli space $\moda{\bk}{n}$, and in particular about its reduced virtual fundamental class. All throughout we assume that $A$ is a generic Abelian surface and has Picard number 1.

As stated before, the Gromov-Witten theory of an Abelian or K3 surface, strictly speaking is trivial, as any of these can be deformed into a non-algebraic surface. To account for that, we look at a reduced theory for these surfaces. For more detail, see \cite{bryan_leung_k3, bryan_leung_abelian, mpt_k3}.

To construct the reduced class on $\moda{2g+2}{n}$, we use the following approach. Let $A$ be a fixed polarized Abelian surface with polarization of type $(1, n)$, and let $B$ be the family of K\"ahler metrics arising from the hyperk\"ahler structure. Note that $B \cong S^2$, the real 2-sphere.

Let $\mathscr{A} \xrightarrow{\pi} B$ be the family of Abelian surfaces over $B$ given by this family of metrics. That is, $A_b = \pi^{-1}(b)$ is $A$ with the K\"ahler structure given by $b$. We can take the fiber-wise quotient to obtain the family $[\mathscr{A}/\pm1] \xrightarrow{\pi} B$ which we use to construct our reduced class. 

\begin{remark}
The family $\mathscr{A}$ is {\em not} an algebraic family. In fact, to work
with this we must leave the algebraic setting and move into the complex
analytic category. However, while it is not algebraic, it is fiber-wise
K\"ahler, and so we are still able to work with Gromov-Witten invariants of
this family.

It is worth noting that the construction of the reduced class for families
of K3 surfaces has been done in \cite{mpt_k3} purely in the algebraic category.
It seems likely that their methods would work similarly to obtain an
algebraic reduced class for the moduli space of maps into an Abelian
surface, and that we should similarly be able to do this for the orbifold
$[A/\pm1]$. We do not however pursue this approach in this work.

In the end we use the notion of the Twistor family as it is a
well-understood and concrete approach. This concrete approach suits us well,
as it permits us to define our invariants with as little pain as possible.
\end{remark}

Note that we have an inclusion $\iota : [A/\pm1] \to [\mathscr{A}/\pm1]$ as one of the fibres. For brevity's sake, define
\[
M = \moda{2g+2}{\beta} \qquad \text{and} \qquad M' = \moda[{[\mathscr{A}/\pm1]}]{2g+2}{\iota_*\beta}.
\]
We have the following lemma.

\begin{lemma}
Let $[A/\pm1]$ be as above, and suppose that $C \subset [A/\pm1]$ is a holomorphic curve. Then the only K\"ahler structure in $B$ that has a holomorphic curve in the class $[C]$ is the original K\"ahler structure for which $C$ is holomorphic.
\end{lemma}

\begin{proof}
Suppose that there are two complex structures $b, b'$ for which there are curves in the class $[C]$ which are holomorphic. We can then lift these to the cover $A \to [A/\pm1]$ to obtain two differing complex structures on $A$ which support curves in a fixed homology class; this contradicts \cite[Lemma 3.4]{bryan_leung_abelian}.
\end{proof}

From this we obtain the following.

\begin{proposition}
The moduli spaces $M$ and $M'$ are isomorphic in the category of complex analytic stacks.
\end{proposition}

\begin{proof}
There is an obvious map $M \to M'$ induced by the inclusion $\iota$. Specifically, for a family
\[
\xymatrix{
\tilde{C} \ar[r]\ar[d] & A \ar[d] \\
C \ar[r]\ar[d] & [A/\pm1] \\
T \ar@/^1pc/[uu]^{s_i}
}
\]
we can compose with the inclusion $\iota$ to obtain
\[
\xymatrix{
\tilde{C} \ar[r]\ar[d] & A \ar[d]\ar@{.>}[r] & \mathscr{A} \ar[d] \\
C \ar[r]\ar[d] & [A/\pm1] \ar@{.>}[r]_\iota & [\mathscr{A}/\pm1]\\
T \ar@/^1pc/[uu]^{s_i}
}
\]
For the reverse direction, note that all holomorphic curves in $[\mathscr{A}/\pm1]$ land in a fixed fibre $[A/\pm1]_b$. This given a diagram
\[
\xymatrix{
\tilde{C} \ar[r]\ar[d] & \mathscr{A} \ar[d] \\
C \ar[r]_f\ar[d] & [\mathscr{A}/\pm1] \\
T \ar@/^1pc/[uu]^{s_i}
}
\]
we note that the map $f$ factors through this fixed fibre $[A/\pm1]_b$. This yields the inverse map $j : M' \to M$.
\end{proof}

We compute the virtual dimension of $M'$ to be
\[
\int_{\iota_*\beta}c_1\big([\mathscr{A}/\pm1]\big) + (1 - g)(\underbrace{\dim [\mathscr{A}/\pm1]}_{=3} - 3) + (2g + 2) - \sum_{i=1}^{2g+2} \underbrace{age(p_i)}_{=1} = 0
\]
and so we have a virtual fundamental class $[M']^{vir}$ in degree zero.

\begin{definition}
We define the {\em reduced virtual fundamental class} on $M$ to be
\[
[M]^{red} = j_*[M']^{vir}.
\]
\end{definition}

We next investigate the structure of the space $\overline{M}_{\bk,n} = \moda{\bk}{n}$. Recall that, for $\bk : A[2] \to \Z_{\geq 0}$, this is the moduli space of genus 0 twisted stable maps into $[A/\pm1]$ such that $\bk(v)$ stacky points have as image $\bk(v)$.

Denote by $\lambda \vdash \bk$ a `multipartition' of $\bk$. That is, a collection of partitions $\lambda^v \vdash \bk(v)$ with parts $(\lambda^v_1, \ldots, \lambda^v_{r_v})$, indexed by $v \in A[2]$. We say that a twisted stable map has partition type $\lambda$ if we can write the source curve as
\[
\Sigma = \Sigma_0 \cup \bigcup_{v \in A[2]} \bigcup_{i = 1}^{r_v} \Sigma_i^v
\]
with each $\Sigma_i^v$ a (potentially nodal) genus 0 curve with $\lambda_i^v$ marked $\Z/2$-points, and where $\Sigma_i^v$ is attached to the curve $\Sigma_0$ at some point, which is will be stacky depending on the parity of $\lambda_i^v$. The map then collapses each $\Sigma_i^v$ to the stacky point in $[A/\pm1]$ corresponding to $v \in A[2]$.

Lastly, denote by $\overline{M}_{\lambda,n}$ the closed substack consisting of those maps with partition type $\lambda$. The main result is the following.

\begin{proposition}\label{prop_disjoint_union}
If the Picard number of $A$ is 1, then
\[
\overline{M}_{\bk, n} = \coprod_{\lambda \vdash \bk} \overline{M}_{\lambda, n}.
\]
\end{proposition}

\begin{proof}
As stated before, each $\overline{M}_{\lambda, n}$ is a closed substack, and so we must show that they are also open in $\overline{M}_{\bk, n}$. We claim that any deformation of the nodes connected a tooth to the handle cannot be smoothed.

We first note that the collection of rational curves (excluding collapsing components) in $[A/\pm1]$ is 0-dimensional. Indeed, if it were not then by looking at the proper transform we would obtain a positive dimensional family of rational curves in $Km(A)$, which cannot exist (See, e.g., \cite{chen_ratl}).

Now, since the Picard number is 1 and the curve class is primitive, the source curve must be the normalization of its image. We will assume for simplicity that the number of teeth on the curve is 1, and that this tooth is itself irreducible. Thus if we were to smooth the node joining this tooth, the resulting source curve must be irreducible.

Consider now a flat family of twisted stable maps into $[A/\pm1]$ over a punctured base $T' = T \setminus\{p\}$.
\[
\xymatrix{
\mathscr{C} \ar[r] \ar[d] & [A/\pm1] \\
T'
}
\]
Since each fibre is the normalization of the image curve in $[A/\pm1]$, this family must be constant. As the moduli spaces $\overline{M}_{\bk, n}$ are separated and proper, there is a unique way to fill in the central fibre, which in this case must also be a constant family. In particular, the resulting central fibre is also the normalization of the image, and so must have no teeth. It follows then that no nodes joining the teeth to the handle can be smoothed.

\end{proof}

The main result of this decomposition is the following. Recall (See Proposition \ref{prop_comb_curve_reduction}) that all curves in $\moda{2g+2}{n}$ are comb curves. We claimed that the components corresponding to even parts contribute zero to the Gromov-Witten invariant---this is equivalent to saying that if a comb curve has a tooth which is joined to the handle at a non-orbifold point, then it contributes zero to the total invariant.

\begin{lemma}\label{lem_even_parts}
Let $v \in A[2]$ be fixed, and let $\overline{M}_{0;2g+2\mid 2k,v}$ denote the component consisting of those twisted stable maps into $[A/\pm1]$ which collapse a component to $v$ with $2k$ marked stacky points, and with $2g+2$ marked stacky points elsewhere. Then
\[
\int_{[\overline{M}_{0;2g+2\mid 2k,v}]^{red}} 1 = 0
\] 
\end{lemma}

\begin{proof}
We will assume without loss of generality (by induction on the number of collapsing components) that the source curve consists of two components, $\Sigma_1$ and $\Sigma_2$ joined at a non-stacky point $P$, and with $\Sigma_1$ being the handle.

We begin with a little notation. Define
\[
M_1 = \modac{2g+2,1}{n}
\]
to be the moduli space of twisted stable maps into $[A/\pm1]$ with no collapsing components and with one ordinary (i.e. non-stacky) marked point. This has (reduced) virtual dimension 1. Similarly, define
\[
M_2 = \moda[B\Z/2]{2k,1}{0}
\]
to be the moduli space of twisted stable maps into $B\Z/2 = [\C^2/\pm1]$, which we think of as the local model for one of the stacky points in $[A/\pm1]$. This also has virtual dimension 1.

Now, unlike the case where the node is stacky, we do not have an isomorphism
\[
M_{0;2g+2\mid 2k,v} \cong M_1 \times M_2
\]
due to the fact that the curves are joined at a non-stacky point. What we do have, however, is a morphism $\iota : M_{0;2g+2\mid 2k,v} \to M_1 \times M_2$ which fits into the gluing diagram
\[
\xymatrix{
M_{0;2g+2\mid 2k, v} \ar[d] \ar[rr]^\iota & & M_1 \times M_2 \ar[d]^{ev \times ev} \\
I[A/\pm1] \ar[rr]_\Delta & & I[A/\pm1] \times I[A/\pm1]
}
\]
where the evaluation maps on the right are from the non-stacky points. From this we obtain (see \cite[Proposition 5.3.1]{agv_long}) that
\[
[M_{0;2g+2\mid 2k, v}]^{red} = \Delta^!\big([M_1]^{red} \times [M_2]^{red}\big).
\]
Now, since the evaluations are at non-stacky points, they in fact lie in the non-twisted sector, which is $[A/\pm1]$. Since $[A/\pm1]$ satisfies Poincar\'e duality (rationally, at least), we can choose a basis $(\gamma_i)$ of the cohomology of $A/\pm1$ so that this is given by
\[
\Delta^!\big([M_1]^{red} \times [M_2]^{red}\big) = \sum_i \int_{ev_*[M_1]^{red}}\gamma_i\int_{ev_*[M_2]^{red}}\gamma^i.
\]
Since $[M_1]^{red}$ and $[M_2]^{red}$ are classes in $H_2$, we see that the only cohomology classes which may contribute are those in dimension 2. However, since the map $ev : M_2 \to [A/\pm1]$ is constant (recall that this is a collapsing component), it doesn't intersect any classes in $H^2$, and so each of the integrals $\int_{ev_*[M_2]^{red}}\gamma^i$ are zero. It follows then that $[M_{0;2g+2\mid 2k,v}]^{red} = 0$ as claimed.
\end{proof}

We will next provide the proofs of several facts concerning the non-enumerative nature of the Gromov-Witten invariants stated in section \ref{sec_gw_hyper}.

First, recall that we have the decomposition of moduli spaces
\[
U_{\lambda, n} = \modac{\bk_\lambda}{n} \times \prod_{v\in A[2]} \prod_{i=1}^{r_v} \moda[B\Z/2]{\lambda_i^v+1}{0}
\]
together with a projection map $p_\lambda : U_{\lambda,n} \to \modac{\bk_\lambda}{n}$.

\begin{theorem}[Theorem \ref{thm_virt_degree}]
Let $\lambda = (\lambda^v)_{v\in A[2]}$ be a collection of partitions of $\bk$, all of which consist of odd parts. Then the virtual degree of $p_\lambda$ is $\big(-\frac{1}{4}\big)^{\frac{1}{2}(|\bk| - |\bk_\lambda|)}$. That is,
\[
(p_\lambda)_*[U_{\lambda,n}]^{red} = \Big(-\frac{1}{4}\Big)^{\frac{1}{2}(|\bk| - |\bk_\lambda|)}\big[\modac{\bk_\lambda}{n}\big]^{red}.
\]
\end{theorem}

The proof of this is virtually identical to that in \cite[Section 3.6]{wise}. Let $\pi : \mathscr{C} \to U_{\lambda, n}$ denote the universal curve, and let $\Sigma = \Sigma_0 \cup \Sigma_1 \cup \cdots \cup \Sigma_k$ denote a comb curve. Since all deformations of the nodes which join the teeth to the handle are obstructed, we have the exact sequence
\begin{equation}\label{eq_exact_sequence_hodge}
0 \to \bigoplus_{P_i} \pi_*(T_{P_i}\Sigma_0 \otimes T_{P_i} \Sigma_i) \to Obs(f) \to Obs(\Sigma, f) \to 0.
\end{equation}
We will compute $Obs(f)$, and use this exact sequence to compute the obstruction bundle $Obs(\Sigma,f)$.

\begin{lemma}
Over a point $[f : \Sigma \to \mathscr{A}/\pm1]$, the bundle $Obs(f)$ is isomorphic to 
\[
Obs(f) \cong H^1\big(\Sigma_0, f^*T[\mathscr{A}/\pm1]|_{\Sigma_0}\big)\oplus\bigoplus_{i=1}^k H^1(\Sigma_i, \rho_1 \oplus \rho_1)
\]
where $\rho_1$ is the non-trivial representation of $\Z/2$.
\end{lemma}

\begin{proof}
As above, let $f: \Sigma \to [\mathscr{A}/\pm1]$ be a comb curve with teeth $\Sigma_i$ for $1 \leq i \leq k$, and let $T = f^*T[\mathscr{A}/\pm1]$. Recall that over such a point that
\[
Obs(f) \cong H^1(\Sigma, T).
\]
To compute this, we look at the normalization sequence
\begin{multline*}
H^0(\Sigma_0, T) \oplus \bigoplus_{i=1}^k H^0(\Sigma_i, T|_{\Sigma_i}) \to \bigoplus_{i=1}^k H^0(P_i, T|_{P_i}) \to \\
\to H^1(\Sigma,T) \to \bigoplus_{i=0}^k H^1(\Sigma_i, T|_{\Sigma_i}) \to 0
\end{multline*}
where $P_i$ is the node joining $\Sigma_i$ to $\Sigma_0$. As $f$ is representable, the image of $P_i$ must lie in the twisted sector, and so we see that $T|_{\Sigma_i} \cong T|_{P_i} \cong \rho_1 \oplus \rho_1 \oplus \rho_0$ where $\rho_1$ and $\rho_0$ denote the non-trivial and trivial representations of $\Z/2$, respectively. Since $\Sigma_i$ has stacky points, and since $P_i \cong B\Z/2$, we have that $H^0(\Sigma_0, \rho_1) \cong 0$ (and similarly for $P_i$). Moreover, it is clear that $H^0(\Sigma_i, \rho_0) \to H^0(P_i, \rho_0)$ is surjective. It follows then that
\[
Obs(f) \cong H^1(\Sigma, T) \cong H^1(\Sigma_0, T|_{\Sigma_0}) \oplus \bigoplus_{i=1}^k H^1(\Sigma_i, T|_{\Sigma_i}).
\]
Since $T|_{\Sigma_i} \cong \rho_1 \oplus \rho_1 \oplus \rho_0$ and $\rho_0$, being the trivial representation, has no higher cohomology, the lemma follows.
\end{proof}

\begin{proof}[Proof of Theorem \ref{thm_virt_degree}]
Since $M_{0;\bk_\lambda}^\circ$ is zero-dimensional, on components of $U_{\lambda,n}$ the summand of $Obs(f)$ coming from $H^1\big(\Sigma_0, f^*T[\mathscr{A}/\pm1]|_{\Sigma_0}\big)$ is a fixed vector space, and so it corresponds to a trivial summand. Moreover, as discussed in \cite{wise_hyper_hodge}, the terms $H^1(\Sigma_i, \rho_1 \oplus \rho_1)$ contributed a dual Hodge bundle summand; that is, $Obs(f) \cong \mathcal{O}^{\oplus d} \oplus \bigoplus_{i=1}^k \mathbb{E}_i^\vee \oplus \mathbb{E}_i^\vee$. Our exact sequence \eqref{eq_exact_sequence_hodge} thus reads
\[
0 \to \bigoplus_{P_i} \pi_*(T_{P_i}\Sigma_0 \otimes T_{P_i} \Sigma_i) \to \mathcal{O}^{\oplus d} \oplus \bigoplus_{i=1}^k \mathbb{E}_i^\vee \oplus \mathbb{E}_i^\vee \to Obs(\Sigma, f) \to 0.
\]
from which we compute that the total Chern class of $Obs(\Sigma,f)$ over $U_{\lambda, n}$ is given by
\[
\underbrace{c(\mathcal{O}^{\oplus d})}_{=1}\prod_{i=1}^k \frac{c(\mathbb{E}_i^\vee)^2}{c\big(\pi_*(T_{P_i}\Sigma_0\otimes T_{P_i}\Sigma_i)\big)}.
\]
We need to integrate this obstruction class over the fibres of the projection map $p_\lambda : U_{\lambda,n} \to \modac{\bk_\lambda}{n}$. Specifically, we need to compute the integrals (using the notation of \cite{gil_unpublished})
\begin{equation}\label{eq_hodge_psi_wise}
\int_{\moda[B\Z/2]{2g_i+2}{0}} \frac{c(\mathbb{E}_i)^2}{1 - \tfrac{1}{2}\psi_1}
\end{equation}
which are computed in \cite{wise_hyper_hodge,gil_unpublished} to be $(-\frac{1}{4})^{g_i}$.

Since the fibre is the product
\[
\prod_{v \in A[2]} \prod_{i=1}^{r_v} \moda[B\Z/2]{\lambda_i^v+1}{0}
\]
it follows that the degree of the pushforward is
\begin{align*}
\prod_{v \in A[2]}\prod_{i=1}^{r_v} \Big(-\frac{1}{4}\Big)^{\tfrac{1}{2}(\lambda_i^v - 1)}
 &= \prod_{v \in A[2]} \Big(-\frac{1}{4}\Big)^{\tfrac{1}{2}\bk(v) - \tfrac{1}{2}r_v} \\
 &= \Big(-\frac{1}{4}\Big)^{\tfrac{1}{2}(|\bk| - |\bk_\lambda|)}
\end{align*}
as claimed.
\end{proof}

%
%

Consider now case that $A \cong S \times F$. The same reasoning as above yields that on any component of $\moda{2g+2}{n}$ which collapses a component between two non-collapsing components (see figure \ref{fig_collapse_between}), the virtual class is obtained by computing the integral
\[
\int_{\moda[B\Z/2]{2g+2}{0}} \frac{c(\mathbb{E})^2}{(1 - \tfrac{1}{2}\psi_1)(1 - \tfrac{1}{2}\psi_2)}
\]
which arises due to the two nodes whose smoothings are obstructed.

\begin{proposition}\label{prop_double_psi}
Let $g > 0$. Then the integral
\[
\int_{\moda[B\Z/2]{2g+2}{0}} \frac{c(\mathbb{E})^2}{(1 - \tfrac{1}{2}\psi_1)(1 - \tfrac{1}{2}\psi_2)}
\]
is zero.
\end{proposition}

\begin{proof}
We follow a method similar to the one given in \cite{gil_unpublished} to prove that \eqref{eq_hodge_psi_wise} is equal to $\left(-\frac{1}{4}\right)^g$. More specifically, we assemble the Gromov-Witten invariants into a generating function, which we will see must be equal to zero.

Let $\mathscr{X} = [\C^2/\pm1]$. Let 1 and $v$, respectively, denote the generators of the untwisted and twisted sectors of $H^*(\mathscr{X})$, and let $\langle \cdots \rangle$ denote the integral
\[
\int_{\overline{M}} \cdots
\]
where the integral is over the appropriate moduli space of genus 0 twisted stable maps into $\mathscr{X}$. Let $g > 1$, and let $a, b$ be non-negative integers. The topological recursion relations in this case yield
\begin{align*}
\langle v^{2g-1}, \tau_{a+1}v, \tau_b v, v\rangle &= 2 \sum_{i=1}^{g-1} {2g-1 \choose 2i} \langle v^{2i}, \tau_a v, v\rangle\langle v^{2(g-i)-1}, \tau_b v, v, v\rangle \\
 & \quad +  2\sum_{i=1}^g {2g-1 \choose 2i-1} \langle v^{2i-1}, \tau_a v, 1\rangle\langle v^{2(g-i)}, \tau_b v, v, 1\rangle \\
 &= 2 \sum_{i=1}^{g-1} {2g-1 \choose 2i} \langle v^{2i+1}, \tau_a v\rangle\langle v^{2(g-i)+1}, \tau_b v \rangle \\
 & \quad +  2\sum_{i=1}^g {2g-1 \choose 2i-1} \langle v^{2i-1}, \tau_{a-1} v\rangle\langle v^{2(g-i)+1}, \tau_{b-1} v\rangle
\end{align*}
where the second equality is given by the string equation. For $g = 1$, this reads
\[
\langle v^2, \tau_{a+1}v, \tau_b v \rangle = 2 \langle v, \tau_a v, 1 \rangle \langle v, \tau_b v, 1\rangle
\]
due to the requirement that each component of the curve have an even number of stacky points.

Multiplying both sides of this by $2^{-a-b-1}$ and summing $a$ and $b$ from 0 to $\infty$ yields
\begin{multline}\label{trr_2}
\left\langle v^{2g}, \frac{v}{1 - \tfrac{1}{2}\psi_1}, \frac{v}{1 - \tfrac{1}{2}\psi_2}\right\rangle - \left\langle v^{2g+1}, \frac{v}{1 - \tfrac{1}{2}\psi_2}\right\rangle \\
= \sum_{i=1}^{g-1} {2g-1\choose 2i} \left\langle v^{2i+1}, \frac{v}{1 - \tfrac{1}{2}\psi_1}\right\rangle\left\langle v^{2(g-i)+1}, \frac{v}{1 - \tfrac{1}{2}\psi_2}\right\rangle \\
+ \frac{1}{4} \sum_{i=1}^g {2g-1\choose 2i-1} \left\langle v^{2i-1}, \frac{v}{1 - \tfrac{1}{2}\psi_1}\right\rangle\left\langle v^{2(g-i)+1}, \frac{v}{1 - \tfrac{1}{2}\psi_2}\right\rangle
\end{multline}

We now assemble these into a generating series. Let $H(q)$ and $h(q)$ be given by
\begin{gather*}
H(q) = \sum_{g=1}^\infty \left\langle v^{2g}, \frac{v}{1 - \tfrac{1}{2}\psi_1}, \frac{v}{1 - \tfrac{1}{2}\psi_2}\right\rangle \frac{q^{2g-1}}{(2g-1)!} \\
h(q) = q + \sum_{g=1}^\infty \left\langle v^{2g+1}, \frac{v}{1 - \tfrac{1}{2}\psi}\right\rangle \frac{q^{2g+1}}{(2g+1)!}
\end{gather*}
From \cite{gil_unpublished}, we see that $h(q) = 2\sin(q/2)$. Moreover, if we multiply \eqref{trr_2} by $\frac{q^{2g-1}}{(2g-1)!}$ and sum from $g = 1$ to $\infty$, we obtain
\begin{multline}\label{trr_4}
H(q) - h''(q) = \\
\sum_{g=1}^\infty \sum_{i=1}^{g-1}\left\langle v^{2i+1}, \frac{v}{1 - \tfrac{1}{2}\psi_1}\right\rangle \frac{q^{2i}}{(2i)!}\left\langle v^{2(g-i)+1}, \frac{v}{1 - \tfrac{1}{2}\psi_2}\right\rangle \frac{q^{2(g-i)-1}}{(2g - 2i - 1)!}\\
+ \frac{1}{4} \sum_{g=1}^\infty \sum_{i=1}^g \left\langle v^{2i-1}, \frac{v}{1 - \tfrac{1}{2}\psi_1}\right\rangle \frac{q^{2i-1}}{(2i-1)!}\left\langle v^{2(g-i)+1}, \frac{v}{1 - \tfrac{1}{2}\psi_2}\right\rangle \frac{q^{2(g-i)}}{(2g-2i)!}
\end{multline}
A somewhat tedious computation yields that the right-hand side of \eqref{trr_4} is equal to
\[
\left(h'(q) - 1\right)h''(q) + \frac{1}{4}h(q)h'(q) = h'(q)h''(q) - h''(q) + \frac{1}{4}h(q)h'(q).
\]
Since $h''(q) = -\frac{1}{4}h(q)$, this is simply equal to $-h''(q)$. It follows then that $H(q) = 0$. Since
\[
\left\langle v^{2g}, \frac{v}{1 - \tfrac{1}{2}\psi_1}, \frac{v}{1 - \tfrac{1}{2}\psi_2}\right\rangle = \int_{\moda[B\Z/2]{2g+2}{0}} \frac{c(\mathbb{E})^2}{(1 - \tfrac{1}{2}\psi_1)(1 - \tfrac{1}{2}\psi_2)}
\]
the claim follows.
\end{proof}

%
%

Next, recall that we defined the generating functions
\begin{gather*}
F_n(z_v) = \sum_{\bk : A[2] \to \zp} GW_{\bk,n} \prod_{v \in A[2]} \frac{z_v^{\bk(v)}}{\bk(v)!} \\
F_n^\circ(x_v) = \sum_{\bk : A[2] \to \zp} GW_{\bk,n}^\circ \prod_{v \in A[2]} \frac{x_v^{\bk(v)}}{\bk(v)!}.
\end{gather*}
We prove now the following.

\begin{theorem}[Theorem \ref{thm_substitution}]
The two generating functions $F_n$ and $F_n^\circ$ are equal after the substitution $x_v = 2 \sin(z_v/2)$.
\end{theorem}

\begin{proof}
We prove this by computing $F_n^\circ\big(2\sin(z_v/2)\big)$, and showing that this is equal to $F_n(z_v)$. For simplicity we use the notation $A^\vee = \map(A[2], \zp)$. Note that
\[
2\sin(z_v/2) = \sum_{\ell=0}^\infty \Big(-\frac{1}{4}\Big)^\ell\frac{z_v^{2\ell+1}}{(2\ell+1)!}.
\]
Substituting this into the definition for $F_n^\circ$ we obtain
\begin{align*}
F_n^\circ\big(2\sin(z_v/2)\big) &= \sum_{\bk \in A^\vee} GW_{\bk,n}^\circ \prod_{v \in A[2]} \frac{1}{\bk(v)!}\Bigg(\sum_{\ell=0}^\infty \Big(-\frac{1}{4}\Big)^\ell\frac{z_v^{2\ell+1}}{(2\ell+1)!}\Bigg)^{\bk(v)} \\
 &= \sum_{\bk \in A^\vee} GW_{\bk,n}^\circ \prod_{v \in A[2]} \sum_{\ell=0}^\infty \Big(-\frac{1}{4}\Big)^\ell\frac{z_v^{2\ell + \bk(v)}}{(2\ell + \bk)(v)!}s\big(2\ell + \bk(v),\bk(v)\big)
\end{align*}
where
\[
s(k,\ell) = \frac{1}{\ell!}\sum_{\substack{a_1 + \cdots + a_\ell = k \\ a_i \text{ odd}}} \binom{k}{a_1, \ldots, a_\ell}.
\]
Exchanging the order of the summation over $\ell$ and the product over $v$ in this expression we find that
\begin{align*}
F_n^\circ\big(2\sin(z_v/2)\big) &= \sum_{\bk \in A^\vee} GW_{\bk,n}^\circ \sum_{\ell \in A^\vee}\Big(-\frac{1}{4}\Big)^{|\ell|} \prod_{v \in A} \frac{z_v^{(2\ell + \bk)(v)}}{(2\ell + \bk)(v)!}s\big((2\ell + \bk)(v),\bk(v)\big) \\
&= \sum_{\bk,\ell \in A^\vee} GW_{\bk,n}^\circ \Big(-\frac{1}{4}\Big)^{|\ell|} \prod_{v \in A} \frac{z_v^{(2\ell + \bk)(v)}}{(2\ell + \bk)(v)!}s\big((2\ell + \bk)(v),\bk(v)\big)
\end{align*}
If we then re-index the summation by letting $\bk' = 2\ell + \bk$ (and for simplicity of notation omitting the $'$), we find
\[
F_n^\circ\big(2\sin(z_v/2)\big) = \sum_{\bk,\ell \in A^\vee} GW_{\bk - 2\ell,n}^\circ \Big(-\frac{1}{4}\Big)^{|\ell|} \prod_{v \in A} \frac{z_v^{\bk(v)}}{\bk(v)!}s\big(\bk(v),(\bk-2\ell)(v)\big)
\]
The claim then that these two generating functions are equal is equivalent then to the claim that
\[
GW_{\bk,n} = \sum_{\ell \in A^\vee} GW_{\bk - 2\ell,n}^\circ \Big(-\frac{1}{4}\Big)^{|\ell|} \prod_{v \in A} s\big(\bk(v), (\bk - 2\ell)(v)\big)
\]
This follows from theorem \ref{thm_virt_degree} and from the following interpretation of the numbers $s(k, \ell)$.

The number $s(k,\ell)$ gives the count of all possible ways of partitioning $k$ marked points into $\ell$ (unordered) odd-sized collections of points.

In our case, by summing over all possible functions $\ell : A[2] \to \zp$, the numbers $s\big(\bk(v), (\bk - 2\ell)(v)\big)$ yield the count of all possible ways of partitioning the $\bk(v)$ points mapping to a given 2-torsion point $v$ by bubbling off collapsing components (all of which must have odd numbers of marked points). Order does not matter as they all map to the same point. For each such possibility, the Gromov-Witten invariant is then $GW_{\bk - 2\ell}^\circ$ (the invariant coming from the non-collapsing component) times $\big(-\frac{1}{4}\big)^{|\ell|}$ (the virtual degree of the map which forgets the collapsing components), as discussed above. This proves the theorem.
\end{proof}

\bibliography{bibliography}

\begin{thebibliography}{BPVdV84}

\bibitem[AGV02]{agv}
Dan Abramovich, Tom Graber, and Angelo Vistoli.
\newblock Algebraic orbifold quantum products.
\newblock In {\em Orbifolds in mathematics and physics ({M}adison, {WI},
  2001)}, volume 310 of {\em Contemp. Math.}, pages 1--24. Amer. Math. Soc.,
  Providence, RI, 2002.

\bibitem[AGV08]{agv_long}
Dan Abramovich, Tom Graber, and Angelo Vistoli.
\newblock Gromov-{W}itten theory of {D}eligne-{M}umford stacks.
\newblock {\em Amer. J. Math.}, 130(5):1337--1398, 2008.

\bibitem[AR10]{andrews_rose}
George~E. Andrews and Simon~C.F. Rose.
\newblock Macmahon's sum-of-divisors functions, {C}hebyshev polynomials, and
  quasi-modular forms.
\newblock {\em J. Reine Angew. Math.}, 2010.
\newblock To appear.

\bibitem[BG09]{crc}
Jim Bryan and Tom Graber.
\newblock The crepant resolution conjecture.
\newblock In {\em Algebraic geometry---{S}eattle 2005. {P}art 1}, volume~80 of
  {\em Proc. Sympos. Pure Math.}, pages 23--42. Amer. Math. Soc., Providence,
  RI, 2009.

\bibitem[BL99]{bryan_leung_abelian}
Jim Bryan and Naichung~Conan Leung.
\newblock Generating functions for the number of curves on abelian surfaces.
\newblock {\em Duke Math. J.}, 99(2):311--328, 1999.

\bibitem[BL00]{bryan_leung_k3}
Jim Bryan and Naichung~Conan Leung.
\newblock The enumerative geometry of {$K3$} surfaces and modular forms.
\newblock {\em J. Amer. Math. Soc.}, 13(2):371--410 (electronic), 2000.

\bibitem[BPVdV84]{bpvdv}
W.~Barth, C.~Peters, and A.~Van~de Ven.
\newblock {\em Compact Complex Surfaces}, chapter VIII.
\newblock Springer-Verlag, 1984.

\bibitem[Che99]{chen_ratl}
X.~Chen.
\newblock Rational curves on {K}3 surfaces.
\newblock {\em J. Algebraic Geom.}, 8:245--278, 1999.

\bibitem[Gil08a]{gillam}
William Gillam.
\newblock {\em Hyperelliptic Gromov-Witten Theory}.
\newblock PhD thesis, Columbia University, 2008.

\bibitem[Gil08b]{gil_unpublished}
William Gillam.
\newblock Letter to {J}. wise, 2008.

\bibitem[G{\"o}t98]{gottsche}
Lothar G{\"o}ttsche.
\newblock A conjectural generating function for numbers of curves on surfaces.
\newblock {\em Comm. Math. Phys.}, 196(3):523--533, 1998.

\bibitem[Gra01]{graber}
Tom Graber.
\newblock Enumerative geometry of hyperelliptic plane curves.
\newblock {\em J. Algebraic Geom.}, 10(4):725--755, 2001.

\bibitem[LB92]{cav}
H.~Lange and Ch. Birkenhake.
\newblock {\em Complex Abelian Varieties}, chapter~5.
\newblock Springer-Verlag, 1992.

\bibitem[Leg28]{leg-1828}
A.~M. Legendre.
\newblock {\em Trait\'e des Fonctions Elliptiques}, pages 133--134.
\newblock Imprimerie de Huzard-Courcier, 1828.

\bibitem[Mac86]{macmahon}
P.~A. MacMahon.
\newblock Divisors of numbers and their continuations in the theory of
  partitions.
\newblock In G.~Andrews, editor, {\em Reprinted: Percy A. MacMahon Collected
  Papers}, pages 305--341. MIT Press, Cambridge, 1986.

\bibitem[MPT10]{mpt_k3}
D.~Maulik, R.~Pandharipande, and R.~P. Thomas.
\newblock Curves on {$K3$} surfaces and modular forms.
\newblock {\em J. Topol.}, 3(4):937--996, 2010.
\newblock With an appendix by A. Pixton.

\bibitem[Wis08a]{wise}
Jonathan Wise.
\newblock {\em The genus zero {G}romov-{W}itten invariants of $[{S}ym^2 \PP^2]$
  and the enumerative geometry of hyperelliptic curves in $\PP^2$}.
\newblock ProQuest LLC, Ann Arbor, MI, 2008.
\newblock Thesis (Ph.D.)--Brown University.

\bibitem[Wis08b]{wise_hyper_hodge}
Jonathan Wise.
\newblock A hyperelliptic hodge integral, 2008.

\bibitem[YZ96]{yau_zaslow}
Shing-Tung Yau and Eric Zaslow.
\newblock B{PS} states, string duality, and nodal curves on {$K3$}.
\newblock {\em Nuclear Phys. B}, 471(3):503--512, 1996.

\end{thebibliography}

\end{document}